\documentclass[12pt]{amsart}
\usepackage{amsfonts}
\usepackage{enumerate}
\usepackage{helvet,courier}
\usepackage{mathptmx,amsmath}
\usepackage{type1cm}
\DeclareMathAlphabet{\mathcal}{OMS}{cmsy}{m}{n}
\DeclareMathAlphabet{\mathbbold}{U}{bbold}{m}{n}  
\usepackage{amsthm}      
\usepackage{graphicx}       
\usepackage{multicol}
\usepackage{mathrsfs}
\usepackage{amssymb}
\usepackage{verbatim}

\usepackage[usenames,dvipsnames]{xcolor}

\usepackage[title]{appendix}

\theoremstyle{plain}
\newtheorem{thm}{Theorem}
\newtheorem{prop}{Proposition}[section]
\newtheorem{lm}[prop]{Lemma}

\newtheorem{conj}{Conjecture}

\theoremstyle{remark}
\newtheorem{rmk}{Remark}

\theoremstyle{definition}

\newcommand{\la}{\lambda}

\newcommand{\bbr}{\mathbb{R}}
\newcommand{\rn}{\mathbb{R}^n}
\newcommand{\bbz}{\mathbb{Z}}

\newcommand{\bbn}{\mathbb{N}}

\newcommand{\sset}{\subset}

\newcommand{\ep}{\epsilon}

\newcommand{\de}{\delta}

\newcommand{\f}{\frac}
\newcommand{\nf}{\infty}

\newcommand{\tf}{\tfrac}
\newcommand{\wh}{\widehat}

\newcommand{\zozo}{[0,\,1]^{2}}
\newcommand{\com}{,\,}

\newcommand{\spt}{\text{supp }}

\begin{document}

\author{Danqing He}
\address{Danqing He, School of Mathematical Sciences,
Fudan University, People's Republic of China}
\email{hedanqing@fudan.edu.cn}

\author{Xinyu Zhu}
\address{Xinyu Zhu, School of Mathematical Sciences,
Fudan University, People's Republic of China}
\email{23110180063@m.fudan.edu.cn}

\thanks{D. He and X. Zhu are supported by  National Key R$\&$D Program of China (No. 2021YFA1002500), NNSF of China (No. 12322105) and Natural Science Foundation of Shanghai (No. 23QA1400300). In addition, D. He is supported by the New Cornerstone Science Foundation. }

\title{A Roth theorem in $\bbr^2$ and a related  ergodic theorem}
\date{}
\maketitle

\begin{abstract}
    We prove a quantitative Roth theorem in the plane for the two-dimensional polynomial pattern $(x_1,x_2), (x_1,x_2)+(t_1,t_2), (x_1,x_2)+(t_1^2+t_2^2,t_1^3+t_2^3)$. A pointwise convergence result for the associated polynomial ergodic average is also obtained. A new bilinear Sobolev improving estimate serves as the primary analytic tool, derived from a new sublevel set estimate.
\end{abstract}

\section{Introduction}

The famous Erd\H{o}s-Tur\'an conjecture \cite{Erdos1936} states that a set $E$ of integers with positive upper Banach density must contain an arithmetic progression of $k$ terms for any $k\in\mathbb N$. Roth \cite{Roth1953} first settled the conjecture for $k=3$, and  Szemer\'edi \cite{Szemeredi1975} later provied a complete proof. Alternative proofs were subsequently given by Furstenberg \cite{Furstenberg1977} and Gowers \cite{Gowers2001} separately. Green and Tao \cite{Green2008} proved the case when $E$ is the set of primes, a breakthrough toward the stronger Erd\H{o}s-Tur\'an conjecture which posits the same for any set $E$ with $\sum_{k\in E}\f1k=\nf$.

The Szemer\'edi Theorem has since become a pivotal result in additive number theory, promoting a wide range of  improvements and generalizations; see, for instance, \cite{Bourgain1999a}, \cite{Sanders2011}, \cite{Bourgain2016a}, \cite{Dong2017}, \cite{Peluse2019}, \cite{Bloom2020}, \cite{Kelley2023},  \cite{Peluse2024}.

Furstenberg, Katznelson, and Weiss \cite{Furstenberg1990} studied an extension of  Szemer\'edi's Theorem to $\bbr^n$, which was also proved by Bourgain \cite{Bourgain1986c} using a Fourier analysis approach. In a related direction,
Bourgain \cite{Bourgain1988b} established a (quantitative) nonlinear Roth theorem, guaranteeing  the existence of the pattern $(x,x+t,x+t^2)$ in sets $E\subset [0,N]$ with positive density.  \cite{Bergelson2000} subsequently generalized this result to a much broader setting, using a Szemer\'edi type result obtained in \cite{Bergelson1996}.

{\noindent \bf Theorem A.} (\cite{Bergelson2000})
{\it  Let $\varepsilon > 0$, $k, l \in \mathbb{N}$, and let $p_{i}(x)=(p_{i,1},\dots, p_{i,l}) \in \big(\mathbb{R}[x]\big)^l$ with $p_{i,j}(0) = 0$, $1 \leq i \leq k$, $1 \leq j \leq l$. Let $d = \max_{1 \leq i \leq k, 1 \leq j \leq l} \deg p_{i,j}$. There exists $\delta > 0$ having the property that for any $N > 1$, if $S \subset [0, N]^l$ is a measurable set with $\lambda(S) \geq \varepsilon N^l$ then there exist $x\in \mathbb{R}^l$ and $t \in \mathbb{R}$ with $t \geq \delta N^{1/d}$ such that 
$
x, x+p_i(t)
 \in S  \qquad \forall 1\le i\le k.
$
}

Building on the $\sigma$-uniformity method introduced by \cite{Li2013}, \cite{Durcik2017} proved a quantitative version of  Theorem A for $p_1(t)=t$  in $\bbr$ when $k=2$.
In \cite{Christ2020c},
research then expanded to the case $p_1(t)=(t,0)$ and $p_2(t)=(0,t^2)$ in $\bbr^2$ by developing the method originated in \cite{Christ2020}.
Further  quantitative results appear in  \cite{Chen2020a}, \cite{Chen2023}, \cite{Chen2024}, \cite{KMPW},  and \cite{Kosz2024}.

A key limitation of these works is that they treat only one-dimensional polynomials. Since Theorem~A applies to polynomials of any dimension, a significant and natural goal is to develop quantitative extensions for higher-dimensional polynomials.

Although the quantitative version of Theorem A for $t\in\bbr$ and $k=2$ remains open, we can study a closely related
pattern $(t_1,t_2)$ and  $(t_1^2+t_2^2,t_1^3+t_2^3)$. 
As many  patterns with two-dimensional polynomials in two variables reduce to  one-dimensional cases (see the Appendix for some examples), the pattern  $(t_1,t_2)$ and  $(t_1^2+t_2^2,t_1^3+t_2^3)$ represents the simplest genuinely  (and hence nontrivial) two-dimensional case, making it a natural starting point.

Let $x=(x_1,x_2)\in\bbr^2$, $t=(t_1,t_2)\in\bbr^2$, $P_1(t)=t_1^2+t_2^2$, $P_2(t)=t_1^3+t_2^3$, 
and $P(t)=(P_1(t),P_2(t))$. Our first main result in this note is as follows.

\begin{thm}\label{t3}
	Let $\ep\in(0,\ \frac{1}{2})$, $N\ge 1$. Then there exists a constant $\de(\ep)\gtrsim e^{-e^{c\ep^{-3}}}>0$ such that the following holds, where $c$ is a constant independent of $\ep$ and $N$. If $E\subseteq [0,\, N^2]\times[0,N^3]$  is a measurable set whose Lebesgue measure is greater than $\ep N^5$, then there exist $$x,\, x+t, x+P(t)\in E$$ with $t_1,\,  t_2>\delta(\ep)N$.
\end{thm}

The essential ingredient to establish Theorem~\ref{t3} is a bilinear Sobolev improving estimate. We present first  a basic version adequate  to prove Theorem~\ref{t3} when $N=1$.

Let $\omega$ be a nonnegative smooth bump function supported in $[1/2\com2]$ with  $\int_\bbr \omega(s)ds=1$. For $l\in\bbn$, we define 
$$
\omega_{l}(x):=2^{2l+4}\omega(2^{l}x_1)\omega(2^{l+4}x_2).
$$
Let 
$$
T_l(f_1, f_2)(x)=\int_{\bbr^2}f_1(x+t)f_2(x+P(t))\omega_{l}(t)dt.
$$

\begin{thm}\label{t4}
Let $l_0\in\bbn$ be a large enough integer.
	There exists $\gamma>0 $ such that,  for all $\la\ge 1$, $l>l_0$, and every Schwartz function $f_2$ with $\spt \wh f_2\sset B(0,\, 2\la)\setminus B(0,\, \la)$, we have
	\begin{equation}
		\|T_l(f_1,f_2)\|_{L^{1}(\zozo)}\le C_{l}\la^{-\gamma}\|f_1\|_{L^2(\bbr^2)}\|f_2\|_{L^2(\bbr^2)}\,,
	\end{equation}
    where $C_{l}\le2^{Cl}$ for some $C>0$.
\end{thm}

To obtain a full version of Theorem~\ref{t3} for arbitrary $N$, we require  a more general result that also enables the study of pointwise convergence of the double ergodic average 
     $$
    A_N(f,\,g)(x)=\frac{1}{N^2}\int_{[0,N]^2}f(T_1^{t_1}T_2^{t_2}x)g(T_1^{P_1(t)}T_2^{P_2(t)}x)dt_1dt_2,
    $$
where $(T_1^{t_1})_{t_1\in\bbr}$, $(T^{t_2}_2)_{t_2\in\bbr}$ are two measure-preserving and commuting flows in  $(X,\,\mathcal{B}(X),\,\mu)$, a probability measure space, such that the mapping
    \begin{equation}\label{e84}
        (x,\,t_1,\,t_2)\in X\times\bbr^2\to T^{t_1}_1T_2^{t_2}x\in X
    \end{equation}
    is measurable.
   
\begin{thm}\label{t8}
    Suppose that $f,\,g\in L^\infty(X)$. 
    Then 
    $$
        \lim_{N\to\infty}A_N(f,\,g)(x)
    $$
    exists $\mu$-almost everywhere in $X$.
\end{thm}

\begin{rmk}
    Our result actually holds when $f,g\in L^p(X)$ for a broader range of $p$. We will not pursue this level of generality in the present work.
\end{rmk}

The multiple ergodic average was introduced by Furstenberg \cite{Furstenberg1977} in his  ergodic-theoretic proof of the Szemer\'edi theorem. Variants of such averages were later studied by Bourgain \cite{Bourgain1990}. A systematic treatment was undertaken by Bergelson and  Leibman \cite{Bergelson1996}, who subsequently formulated the following conjecture in \cite{Bergelson2002}. 

\begin{conj}[{\cite[Section 5.5]{Bergelson2002}}]\label{t090}
    Let $G$ be a nilpotent group of measure preserving transformations of a probability space $(X,\mathcal B,\mu)$. Then for any $T_1,\dots, T_l\in G$, and any $f_1,\dots, f_d\in L^\nf(X)$, the limit of the average
    $$
    \f1N\sum_{n=1}^N\prod_{j=1}^d f_j(T_1^{p_{j,1}(n)}\cdots T_l^{p_{j,l}(n)}x)
    $$
    exists in $L^2$-norm and almost everywhere.
\end{conj}

The $L^2$-norm case of this conjecture was  resolved by \cite{Walsh2012}, while pointwise convergence remains largely open. Significant advances toward the pointwise convergence problem have been achieved  in \cite{Krause2020}, \cite{Ionescu2023},  \cite{Kosz2024}. 
Continuous analogues of Conjecture~\ref{t090} have been studied in \cite{CDKR}, \cite{Frantzikinakis2023} and \cite{k1}.

The proof of Theorem~\ref{t8} closely follows the strategy in \cite{CDKR}; our main contribution is the following bilinear estimate. 

\begin{thm}\label{t27}
    For $l\in\bbn$, we define 
$$
r_{l}(x):=2^{2l}\chi_{[0,2^{-l}]^2}(x).
$$
Let 
$$
T^R_l(f_1, f_2)(x)=\int_{\bbr^2}f_1(x+B_Rt)f_2(x+P(t))r_{l}(t)dt,
$$
	    where $B_Rt=(\frac {t_1 }R,\,\frac{t_2}{R^2})$. There exists $\gamma>0 $ such that, for all $\la\ge 1$, $l\in\bbn$, and every $L^2$ function $f_2$ with $\spt \wh f_2\sset B(0,\, 2\la)\setminus B(0,\, \la)$, we have
	\begin{equation}\label{e95}
		\|T_l^R(f_1,f_2)\|_{L^{1}(\zozo)}\le C_{l}\la^{-\gamma}\|f_1\|_{L^2(\bbr^2)}\|f_2\|_{L^2(\bbr^2)}\,,
	\end{equation}
    where $C_{l}\le2^{Cl}$ for some $C>0$ independent of $R$ and $\la$.
\end{thm}

Unlike   the separated-support assumption on $r_l$ in Theorem~\ref{t4},  the variables
$t_1$ and $t_2$ in $\spt r_l$ here are not separated. This creates some technical obstacles; for instance, the change of variables $s=(t,P(t))$ for $t\in \spt r_l$ in Theorem~\ref{t27} fails to be  invertible. 
To overcome these issues, we must partition the  range of $t$ and track the relevant parameters carefully.

Let $R\ge1$, $A\ge1$. 
Let $r(t)$ be a smooth function which satisfies that for any $t\in\spt r\subseteq\zozo$, $|t_1|>A^{-1}$, $|t_2|>A^{-1}$, $|t_1-t_2|>A^{-1}$ and
    \begin{equation}\label{e91}
        \Big|\det 
        \begin{bmatrix}
            2t_1-\frac{1}{R}&2t_2\\
            3t_1^2&3t^2_2-\frac{1}{R^2}
        \end{bmatrix}\Big|=\big|6t_1t_2^2-6t_1^2t_2-\frac{3t_2^2}{R}-2\frac{t_1}{R^2}+\frac{1}{R^3}\big|\ge A^{-1}
    \end{equation}
    and  $\|\partial^\alpha r\|_\infty\le A^{|\alpha|},\quad \forall\alpha\in\bbn^2.$
    We define
        \begin{equation}\label{e087}
            T(f_1, f_2)(x)=\int_{\bbr^2}f_1\big(x+B_Rt\big)f_2(x+P(t))r(t)dt.
       \end{equation}

We are now ready to state the main technical contribution of this note: a generalization of Theorem~\ref{t4}. We observe that  Theorem~\ref{t4} is a special case of Theorem~\ref{t24} by taking $A=2^l$.
  
\begin{thm}\label{t24}

    There exists $\gamma>0$, $K>0$ such that, for all $\la\ge 1$, $R\ge1$, $A\ge1$ and every $f_2\in L^2(\bbr^2)$ with $\spt \wh f_2\sset B(0,\, 2\la)\setminus B(0,\, \la)$, we have
	\begin{equation}\label{e90}
		\|T(f_1,f_2)\|_{L^{1}(\zozo)}\le C A^K \la^{-\gamma}\|f_1\|_{L^2(\bbr^2)}\|f_2\|_{L^2(\bbr^2)}\,,
	\end{equation}
    where the constant $C$ is independent of $R\ge 1$.

\end{thm}

The Fourier multiplier related to the bilinear operator $T$ is 
$$
\int_{\bbr^2}r(t)e^{2\pi i(\xi,\eta)\cdot (B_Rt,P(t))}dt.
$$
The two dimensional structure of this problem introduces essential difficulties. Firstly, the critical points of the phase function $(\xi,\eta)\cdot (B_Rt,P(t))$ are more complicated than those of its one dimensional analogues. In particular, we cannot assume that $\spt\wh f_1\subset B(0,\, 2\la)\setminus B(0,\, \la)$ when $R=1$; consequently  the $\sigma$-uniformity approach used in \cite{Durcik2017} and \cite{Chen2020a} is not easily adapted  to prove \eqref{e90}.

We will follow the strategy in \cite{Christ2020c}, though certain advantages available in that setting do not hold for our operator $T$. First, the useful reduction from $R\ge 1$ to  $R=1$ via a corner structure is not applicable. Second, we lack the gain in the number of  functions pieces available in \cite{Christ2020c}; instead, we explore useful restrictions on the support of the bilinear operator, detailed in Lemma~\ref{t30}. 

Similar to \cite{Christ2020c}, we ultimately reduce Theorem~\ref{t24} to a sublevel set estimate. The proof of this estimate is one of the  main technical contributions of this note.

This paper is organized as follows. In Section~\ref{s4}, we prove Theorem~\ref{t3} and Theorem~\ref{t8}. In Section~\ref{s3}, we prove Theorem~\ref{t27} and  reduce Theorem~\ref{t24} to two estimates on $T_\flat$ and $T_\sharp$ via  a structural decomposition. These two estimates are proved in Section~\ref{s4'}
 and Section~\ref{s5'} respectively.

\medskip

{\noindent \bf Notations.}
In $\bbr^n$ we use $B(a,r)$ to denote the ball centered at $a$, with radius $r$. For a function $f$ and a nonnegative function $g$, we say $|f|\lesssim g$ if there exists $c>0$ such that $|f|\le cg$. For a function $f$ and a nonnegative function $g$, we say $|f|\lesssim_l g$ if for each $l$ there exists $c_l>0$ such that $|f|\le c_lg$. Given $E\sset\rn$, we denote by  $\chi_E$ the characteristic function of $E$, i.e. $$\chi_E(x)=\begin{cases}

     1,\, x\in E  \\
     0,\,x\notin E. \end{cases}
$$
$E^c:=\bbr^n\setminus E$ is the complement of $E$.
We call $f\in C^\infty_0(\bbr^n)$ if $f\in C^\infty(\bbr^n)$ and $f$ has  compact support. We define $Q(a,l)=\{x:\|x-a\|_\infty=l/2\}.$ $x\in a+O(r)$ if $\|x-a\|_\infty=O(r).$

\section{Proof of the applications}\label{s4}

\subsection{A Roth theorem}

By a standard argument (see \cite{Christ2020c}, \cite{Bourgain1988b}, \cite{Durcik2017}, and \cite{KMPW}), we sketch below how we can obtain Theorem~\ref{t3} by assuming Theorem~\ref{t27}.

\begin{proof}[Proof of Theorem~\ref{t3}]

For $t=(t_1,t_2)\in\bbr^2$, we define
$$
\tilde \chi_{t_1,t_2}(x)=\frac{1}{4t_1t_2}\chi_{[-t_1,t_1]\times[-t_2,t_2]}(x),
$$
and
$$
Q_{t}f(x):=\tilde \chi_t\ast f(x)=\f{1}{4t_1t_2}\int_{\begin{matrix}
    |u_1-x_1|\le t_1\\|u_2-x_2|\le t_2
\end{matrix}}f(u)du.
$$ 
We first present an auxiliary result. The one-dimensional version was obtained by \cite{Bourgain1988b}.
\begin{lm}\label{l1}
	Suppose that $f(x)\ge 0$ and $0<t_{1,i}\le t_{2,i}<1$ for $i=1,2$, where $t_j=(t_{1,j},\,t_{2,j})$, $j=1,2$. Then we have
	\begin{equation}\label{e000}
		\int_{\zozo}f(x){Q}_{t_1}f(x)Q_{t_2}f(x)dx
        \ge c \Big(\int_{\zozo}f(x)dx\Big)^{3},
	\end{equation}
    where $c>0$ is an absolute constant independent of $f$ and $t_1,\,t_2$. 
\end{lm}

To prove Theorem~\ref{t3}, we need only to prove that, for $N=2^{n_1}$, where $n_1\in \bbn$ and any measurable function $g$ on $[0,\,N^2]\times[0,\,N^3]$ satisfying $0\le g\le 1$ and $\int_{\bbr^2} g(x)dx\ge\ep N^5$, we have 
$$
\int_{\bbr^2}\int_{[0,N^2]\times[0,N^3]}g(x)g(x+t)g(x+P(t))dxdt\ge 10 \delta(\ep) N^7.
$$

Let $f(x)=g(N^2x_1,\,N^3x_2)$. $f$ is supported in $\zozo$. $f$ satisfies $0\le f\le 1$ and $\int_{\bbr^2} f(x)dx\ge\ep$. We only need to prove that
\begin{equation}\label{e94}
    \int_{\bbr^2}\int_{\zozo}f(x)f(x+B_Nt)f(x+P(t))dxdt\ge 10 \delta(\ep).
\end{equation}
Fix integers $1<k<k'<k''$, which will be determined later. Let
$$
I=\int_{\bbr^2}\int_{\zozo}f(x)f(x+B_Nt)f(x+P(t))dxdt.
$$
Recall that 
$$
r_{l}(x):=2^{2l}\chi_{[0,2^{-l}]^2}(x).
$$ 
We will estimate 
$$
2^{2k'}I\ge\int_{\bbr^4}f(x)f(x+B_Nt)f(x+P(t))r_{k'}(t)dxdt=I_1+I_2+I_3,
$$
where
\begin{align*}
    &I_1=\int_{\bbr^4}f(x)f(x+B_Nt)Q_{(2^{-k},2^{-k})}f(x+P(t))r_{k'}(t)dxdt,\\
    &I_2=\int_{\bbr^4}f(x)f(x+B_Nt)\big(Q_{(2^{-k''},2^{-k''})}f-Q_{(2^{-k},2^{-k})}f\big)(x+P(t))r_{k'}(t)dxdt,\\
    &I_3=\int_{\bbr^4}f(x)f(x+B_Nt)\big(f-Q_{(2^{-k''},2^{-k''})}f\big)(x+P(t))r_{k'}(t)dxdt.
\end{align*}
In the frequency side, we decompose $f-Q_{(2^{-k''},2^{-k''})}f$  into dyadic pieces and for each piece we apply Theorem~\ref{t27}, then there exists $\gamma'>0$ such that
$$
|I_3|\lesssim2^{Ck'-\gamma' k''},
$$
where $\gamma'$ is taken to be smaller than $\tf12\min\{C,1\}$.
By the Cauchy-Schwarz inequality, we also have
$$
|I_2|\le\|Q_{(2^{-k''},2^{-k''})}f-Q_{(2^{-k},2^{-k})}f\|_2.
$$
By Lemma~\ref{l1}, we can see, for $k'>k+5$,
$$
I_1\gtrsim\int_{\bbr^4}f(x)Q_{(2^{-k'-n_1},2^{-k'-2n_1})}f(x)Q_{(2^{-k-1},2^{-k-1})}f(x)dxdt\gtrsim\ep^3.
$$
Collecting all above estimates, we obtain
\begin{align*}
2^{2k'}I
\gtrsim&\ep^3-C_2(\|Q_{(2^{-k''},2^{-k''})}f-Q_{(2^{-k},2^{-k})}f\|_2+2^{Ck'-\gamma' k''}).
\end{align*}

Taking $k_0=5$, and  let $k_n'=k_n+10$ and $k_{n+1}=k_n''=\frac{C}{\gamma'}k_n'-\frac{\log_2(\ep^3/5C_2)}{\gamma'}$, and using, for any $N_0\in\bbz^+$ and $\xi\in\bbr^2$,
$$\sum_{n=0}^{N_0}|\widehat{\tilde \chi}_{(2^{-k_n},2^{-k_n})}(\xi)-\widehat{\tilde \chi}_{(2^{-k_{n+1}},2^{-k_{n+1}})}(\xi)|\lesssim1,
$$
we find by pigeonholing that there exists $n_0\lesssim\frac{1}{\ep^3}$ such that
$$
\|Q_{(2^{-k_{n_0}},2^{-k_{n_0}})}f-Q_{(2^{-k_{n_0+1}},2^{-k_{n_0+1}})}f\|_2\le \frac{\ep^3}{5}.
$$
It follows
that
$$
I\ge 2^{-2k_{n_0}'}\ep^3
\gtrsim e^{-e^{c_1\ep^{-3}}}
$$
for some absolute constant $c_1>0$.
This finishes the proof.

\end{proof}

\subsection{An Ergodic theorem}

In this subsection, we will reduce Theorem~\ref{t8} to Theorem~\ref{t27}, following the strategy in \cite{CDKR} closely.
   
The following two lemmas are standard; see, for instance, \cite{CDKR}
for related proofs.

\begin{lm}\label{t14}
    Suppose that for any $\alpha>1$, we have $\lim\limits_{n\to+\infty}A_{\alpha^n}(f,\,g)(x)$ exists $\mu$-almost everywhere in $X$. Then $\lim\limits_{N\to+\infty}A_N(f,\,g)(x)$ exists $\mu$-almost everywhere in $X$.
\end{lm}

\begin{lm}\label{t15}
    Let $U_i^\delta(f)(x)=f(T_i^\delta x)$ for $\delta\in\bbr$ and $i=1,\,2$ and
    \begin{align*}
        &A=\bigcap_{\delta\in(0,1]}\big(Ker(U_1^{\delta}-I)\cap Ker(U_2^{\delta}-I)\big),\\
        &B=span(\bigcup_{\delta\in(0,1]}\big(\big(Im(U_1^{\delta}-I)\cup Im(U_2^{\delta}-I)\big)\big).
    \end{align*}
    Here $I$ represents the identity mapping. Then $A+B$ is dense in $L^2(X)$.

\end{lm}

We need also the boundedness of bilinear maximal ergodic operators.
\begin{prop}\label{t40}
    Let $\mathcal{P}(t)=\big(\mathcal{P}_1(t),\,\mathcal{P}_2(t),\,\mathcal{P}_3(t),\,\mathcal{P}_4(t)\big)$, where $\mathcal{P}_i(t)$ is given polynomial,  $i=1,\,2,\,3,\,4$. Let
     $$
    \mathcal A^\mathcal{P}_N(f,\,g)(x)=\frac{1}{N^2}\int_{[0,N]^2}f(T_1^{\mathcal{P}_1(t)}T_2^{\mathcal{P}_2(t)}x)g(T_1^{\mathcal{P}_3(t)}T_2^{\mathcal{P}_4(t)}x)dt.
    $$
     For $p_1,p_2\in[1,\infty]$, and $\tf1p=\tf1{p_1}+\tf1{p_2}<1$, we have that  
     $$
    \|\sup\limits_{N}|\mathcal A^{\mathcal{P}}_N(f,\,g)|\|_{L^{p}(X)}\lesssim_{\mathcal{P},p_1,p_2}\|f\|_{L^{p_1}(X)}\|g\|_{L^{p_2}(X)}.
    $$
\end{prop}
\begin{proof}
It suffices to consider the case $f,g\ge 0$.
    Let
   \begin{align*}
        &\mathcal A'_N(f)(x)=\frac{1}{N^2}\int_{[0,N]^2}f(T_1^{\mathcal  P_1(t)}T_2^{\mathcal  P_2(t)}x)dt,\\
        &\mathcal A^{''}_N(f)(x)=\frac{1}{N^2}\int_{[0,N]^2}f(T_1^{\mathcal  P_3(t)}T_2^{\mathcal  P_4(t)}x)dt.
    \end{align*}
     By  Theorem 1.5 in \cite{k1}, we have that for $q\in(1,\infty]$,
     \begin{align*}
         &\|\sup_{N\in\bbr^+}\mathcal A^{'}_N(f)\|_q\lesssim_{q}\|f\|_q,\\
         &\|\sup_{N\in\bbr^+}\mathcal A^{''}_N(f)\|_q\lesssim_{q}\|f\|_q.
     \end{align*}
We remark that here we can take $q=\infty$ as $\mu$ is a probability measure.
     
     Take $q_1=p_1/p$ and $q_2=p_2/p$. Then $1=\frac{1}{q_1}+\frac{1}{q_2}$. Combining this with the H\"older inequality, we have
    \begin{align*}
        &\|\sup\limits_{N\in\bbr^+}|\mathcal A_N(f,\,g)|\|_{L^{p}(X)}^p\\
        =&\int_X\sup\limits_{N\in\bbr^+}|\frac{1}{N^2}\int_{[0,N]^2}f(T_1^{\mathcal  P_1(t)}T_2^{\mathcal  P_2(t)}x)g(T_1^{\mathcal  P_3(t)}T_2^{\mathcal  P_4(t)}x)dt|^pdx\\
        \lesssim&\int_X\sup\limits_{N\in\bbr^+}|\frac{1}{N^2}\int_{[0,N]^2}[f(T_1^{\mathcal  P_1(t)}T_2^{\mathcal  P_2(t)}x)]^{q_1}dt|^{\frac{p}{q_1}}\\
        &\sup\limits_{N\in\bbr^+}|\frac{1}{N^2}\int_{[0,N]^2}[g(T_1^{\mathcal  P_3(t)}T_2^{\mathcal  P_4(t)}x)]^{q_2} dt|^{\frac{p}{q_2}}dx\\
        \lesssim&\|\sup\limits_{N\in\bbr^+}\mathcal A^{'}_N(f^{q_1})\|_p^{\frac{p}{q_1}}\|\sup\limits_{N\in\bbr^+}\mathcal A^{''}_N(g^{q_2})\|_p^{\frac{p}{q_2}}\\
        \lesssim_p&\|f\|_{p_1}^p\|g\|_{p_2}^p,
    \end{align*}
    which finishes the proof.
\end{proof}

We can obtain a similar result in the discrete setting through Theorem~1.20 in \cite{Mirek2022}.
\begin{prop}
    Let $\mathcal{P}(t)=\big(\mathcal{P}_1(t),\,\mathcal{P}_2(t),\,\mathcal{P}_3(t),\,\mathcal{P}_4(t)\big)$, where $\mathcal{P}_i(t)$ is a given polynomial for any $i=1,\,2,\,3,\,4$ that satisfies $\mathcal{P}_i(0)=0$. Let $(X;\,B(X),\,\mu)$ be a $\sigma$-finite measure space endowed with a family $T_1,\, T_2$ of commuting invertible measure-preserving transformations on $X$. Let
     $$
    \mathcal A^\mathcal{P}_N(f,\,g)(x)=\frac{1}{N^2}\sum_{t_1=0}^{N-1}\sum_{t_2=0}^{N-1}f(T_1^{\mathcal{P}_1(t)}T_2^{\mathcal{P}_2(t)}x)g(T_1^{\mathcal{P}_3(t)}T_2^{\mathcal{P}_4(t)}x).
    $$
     For $p_1,p_2\in(1,\infty)$, and $\tf1p=\tf1{p_1}+\tf1{p_2}<1$, we have that  
     $$
    \|\sup\limits_{N\in\bbz^+}|\mathcal A^{\mathcal{P}}_N(f,\,g)|\|_{L^{p}(X)}\lesssim_{\mathcal{P},p_1,p_2}\|f\|_{L^{p_1}(X)}\|g\|_{L^{p_2}(X)}.
    $$
\end{prop}

\begin{proof}[Proof of Theorem~\ref{t8}]
We will actually show that 
\begin{equation}\label{edouble}
\lim_{N\to\infty}A_N(f,\,g)(x)\quad \text{exists a.e. }
\end{equation}
when $f\in L^\infty$ and $g\in L^2$, which obviously implies the conclusion as $L^\nf(X)\sset L^2(X)$.

By taking 
$$
    \mathcal  P_1(t)=t_1,\quad \mathcal  P_2(t)=t_2,\quad \mathcal  P_3(t)=P_1(t),\quad \mathcal  P_4(t)=P_2(t).
    $$
     in Proposition~\ref{t40}, it suffices to verify  \eqref{edouble} when $f\in L^\nf(X)$ and $g\in A+B$ via a standard argument (using, for instance, a modification of \cite[Theorem 2.1.14]{Grafakos2014b}), as $A+B$ is dense in $L^2(X)$ by Lemma~\ref{t15}. By the linearity of $A_N(f,\,g)$, we only need to check the cases $g\in A$ and $g\in B$.

When $g\in A$, by the definition of $A$, the set 
$$
\{(x,t)\in X\times (\bbr_+)^2:\ U_1^{P_1(t)}U_2^{P_2(t)}g(x)\neq g(x)\}
$$
has measure $0$, which implies that, for $\mu$-a.e. $x$, we have
 $$
 A_N(f,g)(x)=\mathcal A'_N(f)(x)g(x)\quad\quad\forall N>0,
 $$  
 where we take $P_1(t)=t_1$ and $P_2(t)=t_2$ in the definition of $\mathcal A'_N$. By Theorem 1.5 in \cite{k1}, for any $f\in L^{\infty}$, the limit $\lim\limits_{N\to+\infty}\mathcal A'_N(f)(x)$ exists $\mu$-almost everywhere in $X$. Therefore \eqref{edouble} holds.

When $g\in B$, due to linearity, we only need to consider the case $g\in Im(T_1^{\delta}-I) $ or $g\in Im(T_2^{\delta}-I) $ where $\delta\in[0,1)$. Without loss of generality, we assume $g\in Im(T_1^{\delta}-I) $, i.e., there exist $h\in L^2(X)$ and $\delta\in [0,1)$, such that for any $t\in\bbr$, $x\in X$, we have
\begin{equation}\label{e85}
    g(T_1^tx)=h(T_1^{t+\delta}x)-h(T_1^tx).
\end{equation}
Let 
\begin{align*}
    &\widetilde{A_N}(f,\,h)(x)\\
    =&\frac{1}{N^2}\int_{[0,N]^2}f(T_1^{t_1}T_2^{t_2}x)\big(h(T_1^{P_1(t)+\delta}T_2^{P_2(t)}x)-h(T_1^{P_1(t)}T_2^{P_2(t)}x)\big)dt_1dt_2,
\end{align*}
for which we have the following result.

\begin{prop}\label{t20}
    For any $\delta>0$, there exists $\gamma=\gamma(\de)>0$,   such that
    \begin{equation}\label{e86}
        \|\widetilde{A_N}(f,\,h)\|_{1}\lesssim N^{-\gamma}\|f\|_2\|h\|_2
    \end{equation}
    holds.
\end{prop}
We take this result for granted momentarily, whose proof will be given  later.

Because of \eqref{e85}, we have $A_N(f,\,g)=\widetilde{A_N}(f,\,h)$. 
Then, for any $\alpha>1$, we have
$$
\int_X\sum_{n=0}^{\infty}|A_{\alpha^n}(f,\,g)(x)|d\mu\lesssim\sum_{n=0}^{\infty}\alpha^{-n\gamma}\|h\|_2\|f\|_2<\infty,
$$
which implies that $\sum_{n=0}^{\infty}|A_{\alpha^n}(f,\,g)(x)|<\infty$ for $\mu$-almost everywhere $x$. 
Then we have 
$$
\lim\limits_{n\to+\infty}A_{\alpha^n}(f,\,g)(x)=0
$$
for $\mu$-almost everywhere $x$. Then by Lemma~\ref{t14}, we finish the proof.
\end{proof}

\begin{proof}[Proof of Proposition~\ref{t20}]
For $x\in\bbr^2$, let
\begin{align*}
    &C_N(F,G)(x)\\
=&\frac{1}{N^2}\int_{[0,N]^2}F(x+t)\big(G(x_1+P_1(t)+\delta,\,x_2+P_2(t))-G(x_1+P_1(t),\,x_2+P_2(t))\big)dt.
\end{align*}
By Calder\'on's transference principle \cite{C68}, the study of \eqref{e86} is reduced to 
prove
    \begin{equation}\label{e87}
        \|C_N(F,\,G)\|_{L^1([0,N^2]\times[0,N^3])}\lesssim N^{-\gamma}\|F\|_2\|G\|_2
    \end{equation}
        uniformly in $\de\in(0,1]$.

Let $\Delta_\delta G(x_1,\,x_2)=G(x_1+\delta,\,x_2)-G(x_1,\,x_2)$, $P(t)=(P_1(t),\,P_2(t))$ and
$$
    \widetilde T(f_1, f_2)(x)=\int_{\zozo}f_1\big(x+B_Nt\big)f_2(x+P(t))dt.
    $$

    Let $f_1(x_1,\,x_2)=N^{5/2}F\big(N^2x_1,\,N^3x_2\big)$, $ f_2(x_1,\,x_2)=N^{5/2}G\big(N^2x_1,\,N^3x_2\big)$. By the change of variables, 
    \eqref{e87} follows from the estimate
    \begin{equation}\label{e89}
        \|\widetilde T(f_1,\Delta_{\delta}f_2)\|_{L^1([0,1]^2)}\lesssim\delta^{\gamma'}\|f_1\|_2\|f_2\|_2.
    \end{equation}
    Let $f_2^0=(\chi_{[0,\delta^{-1/3}]}(|\cdot|)\wh f_2)^{\vee}$ and $f_2^j=(\chi_{[2^{j-1}\delta^{-1/3},2^j\delta^{-1/3}]}(|\cdot|)\wh f_2)^{\vee}$ for $j\in\bbz^+$. Then 
    $
    f_2=\sum_{j=0}^{+\infty}f_2^j.
    $
  Applying Theorem~\ref{t27} with $l=1$,  we can find that
    $$
    \|\widetilde T(f_1,\Delta_\delta f_2^j)\|_{L^{1}(\zozo)}\lesssim2^{-\gamma j}\delta^{\gamma/3}\|f_1\|_{L^2(\bbr^2)}\|f_2\|_{L^2(\bbr^2)}.
    $$
    By the Bernstein inequality, we get
    $$
    \|\Delta_\delta f_2^0\|_\infty\lesssim\delta^{2/3}\|f_2^0\|_\infty\lesssim\delta^{1/3}\|f_2^0\|_2\lesssim\delta^{1/3}\|f_2\|_2,
    $$
    which implies
    \begin{align*}
         &\|\widetilde T(f_1,\Delta_\delta f_2^0)\|_{L^{1}(\zozo)}
         \lesssim\delta^{1/3}\|f_1\|_{L^2(\bbr^2)}\|f_2\|_{L^2(\bbr^2)}.
    \end{align*}
    Summing over $j$, we  get  \eqref{e89}  with $\gamma'=\gamma/3$.
    
\end{proof}

\section{Main reductions}\label{s3}

We start this section by deriving Theorem~\ref{t27} from Theorem~\ref{t24}.

\begin{proof}[Proof of Theorem~\ref{t27}]
    Let $\tilde r_B(s)\in C^\infty(\bbr)$ be supported in $[B,\, +\infty)$ and satisfy $\tilde r_B(s)=1$ when $s\in[2B,+\infty)$ and  $\|\partial^\alpha \tilde r_B\|_\infty\le B^{-|\alpha|}\ \ \forall \alpha\in \bbn$. Let 
\begin{align*}
    r(t)=&\tilde r_B(|t_1|^2)\tilde r_B(|t_2|^2)\tilde r_{2^{-2l}B}(2^{-2l}-|t_1|^2)\tilde r_{2^{-2l}B}(2^{-2l}-|t_2|^2)\\
    &\tilde r_B(|6t_1t_2^2-6t_1^2t_2-\frac{3t_2^2}{R}-2\frac{t_1}{R^2}+\frac{1}{R^3}|^2)\tilde r_B(|t_1-t_2|^2)\in C^\nf(\bbr^2),
\end{align*}
which is designed to a smooth approximation of $\chi_{[0,2^{-l}]^2}$ satisfying Theorem~\ref{t24} with $B=\la^{-2\ep}$ and $A=2^{2l}\la^{2\ep}$. Here $\ep\le\frac{\gamma}{100K}$ with $K>0$ in Theorem~\ref{t24}.
Therefore,
\begin{align}
    &\|\int_{\bbr^2}f_1\big(x+B_Rt\big)f_2(x+P(t))r(t)dt\|_{L^{1}(\zozo)}\notag\\
    \lesssim &2^{2Kl}\la^{-\gamma/2}\|f_1\|_{L^2(\bbr^2)}\|f_2\|_{L^2(\bbr^2)}.\label{e92}
\end{align}

Next we show that the support of $\chi_{[0,2^{-l}]^2}(t)-r(t)$ is relatively small. Observe that $\chi_{[0,2^{-l}]^2}(t)-r(t)$ is supported in $\Omega=\cup_{i=1}^{4}\Omega_i$ with
\begin{align*}
    &\Omega_1=\{t\in\zozo:\,\min_{i\in\{1,2\}}|t_i|^2\le 2B\}\\
    &\Omega_2=\{t\in[0,2^{-l}]^2:\,\max_{i\in\{1,2\}}|t_i|^2\ge2^{-2l}-2^{-2l+1}B\}\\
    &\Omega_3=\{t\in\zozo:\,|t_1-t_2|^2\le 2B\}\\
    &\Omega_4=\{t\in\zozo:\,|6t_1t_2^2-6t_1^2t_2-\frac{3t_2^2}{R}-2\frac{t_1}{R^2}+\frac{1}{R^3}|^2\le2B\}.
\end{align*}
We claim that $|\Omega|\lesssim \lambda ^{-\ep/8}$.
Notice that $|\Omega_1|\lesssim\la^{-\ep/2}$, $|\Omega_2|\lesssim2^{-l/2}\la^{-\ep/2}$ and $|\Omega_3|\lesssim\la^{-\ep/2}$. For $|t_2|\ge\la^{-\ep/4}$, we have 
$$
\big|\{t_1:\,|6t_1t_2^2-6t_1^2t_2-\frac{3t_2^2}{R}-2\frac{t_1}{R^2}+\frac{1}{R^3}|\lesssim\la^{-\ep/2}\}\big|\lesssim\la^{-\ep/8},
$$
which implies $|\Omega_4|\lesssim\la^{-\ep/8}$. So we have that $|\Omega|\lesssim\la^{-\ep/8}$. 

By the H\"older inequality, we have
\begin{align*}
    &\|\int_{\zozo}f_1\big(x+B_Rt\big)f_2(x+P(t))( \chi_{[0,2^{-l}]^2}(t)-r(t))dt\|_{L^{1}(\zozo)}\\
    \le&\int_{\Omega}\int_{\zozo}|f_1\big(x+B_Rt\big)||f_2(x+P(t))|dxdt\\
    \le&\|f_1\|_{L^2(\bbr^2)}\|f_2\|_{L^2(\bbr^2)}|\Omega|\\
    \lesssim&\la^{-\ep/8}\|f_1\|_{L^2(\bbr^2)}\|f_2\|_{L^2(\bbr^2)}.
\end{align*}
Combining this with \eqref{e92} we can get \eqref{e95} with the new $\gamma=\ep/8$.
\end{proof}

We will focus on the proof of Theorem~\ref{t24} below.
In the proof of Theorem~\ref{t24}, we can assume $A\le \la^{\ep'}$, where $\ep$ is a small constant which will be determined later. When $A\ge\la^{\ep'}$, we take $K>\frac{\gamma}{\ep'}$, then we can have \eqref{e90} by the trivial estimate.
We show first that, to prove Theorem~\ref{t24}, it suffices to verify the following result.

\begin{thm}\label{t25}
For $T$ defined in \eqref{e087}, there exists $\sigma>0,\, K>0$ such that
\begin{equation}
	\|T(f_1,\, f_2)\|_{L^{1}(\zozo)}\lesssim A^K\lambda^{-\sigma }\|f_1\|_{L^\infty}\|f_2\|_{L^\infty}\, \label{e5}
\end{equation}
for any $f_1$ and $f_2$ satisfying
$\spt \wh f_2\subset B(0,2\la)\setminus B(0,\la)$
and 
\begin{equation}\label{e004}
    \spt \wh f_1\sset Q_{R,\lambda} :=(-20R\la,\, 20R\lambda)\times(-20R^2\la,\,20R^2\la).
\end{equation}
\end{thm}

We prove Theorem~\ref{t24} by taking Theorem~\ref{t25} as granted momentarily.

\begin{proof}[Proof of Theorem~\ref{t24}]

We explain first why we can assume \eqref{e004}, i.e. $\spt \wh f_1\sset Q_{R,\lambda}$.

Let $f_{1,1}=(\hat{f_1}\chi_{Q_{R,\lambda}})^{\vee}$ and $f_{1,2}=(\hat{f_1}\chi_{\bbr^2\setminus Q_{R,\lambda}})^{\vee}$. 
Let $h$ be a function supported in $\zozo$ satisfying $\|h\|_{L^{\infty}}\le 1$. Notice that 
\begin{align}
	&\int_{\bbr^2}T(f_{1,2},\, f_2)(x)h(x)dx\nonumber\\
	=&\int_{\bbr^4}f_{1,2}\big(x+B_R t\big)f_2(x+P(t))r(t)h(x)dtdx\nonumber\\
	=&\int_{\bbr^4}\widehat{f_{1,2}}(\xi)\hat{f_2}(\eta)\hat{h}(-\xi-\eta) \int_{\bbr^2} r(t)e^{2\pi i\big(B_R t\cdot \xi+P(t)\cdot \eta\big)}dt d\xi d\eta.
    \label{e9}
\end{align}
For  $t_1,\, t_2\in [0,\,1]$, 
we have either $\big|\xi_1/R+2t_1\eta_1+3t_1^2\eta_2\big|\ge 10\lambda$  or $\big|\xi_2/R^2+2t_2\eta_1+3t_2^2\eta_2\big|\ge 10\lambda$ when $\xi\in\bbr^2\setminus Q_{R,\lambda}$. 
Applying the non-stationary phase, we have either $$\Big|\int_{\bbr}e^{2\pi i(\frac{t_1\xi_1}{R}+t_1^2\eta_1+t_1^3\eta_2)}r(t)dt_1\Big|\lesssim A^2\lambda^{-2}$$
or 
$$
\Big|\int_{\bbr}e^{2\pi i(\frac{t_2\xi_2}{R^2}+t_2^2\eta_1+t_2^3\eta_2)}r(t)dt_2\Big|\lesssim A^2\lambda^{-2}
$$ 
recalling that $\|\partial^\alpha r\|_\nf\le A^{|\alpha|}$.
Plugging these estimates into \eqref{e9}, we have 
\begin{align*}
	&\|T(f_{1,2},\, f_2)\|_{L^1(\zozo)}\\
    =&\sup_{h:\ \|h\|_{L^\nf([0,1]^2)}=1}\Big|\int_{\bbr^2}T(f_{1,2},\, f_2)(x)h(x)dx\Big|\\
    \lesssim  & A^2 \lambda^{-1}\|f_{1,2}\|_{L^2}\|f_2\|_{L^2},
\end{align*}
as $\|h\|_{L^2}\le \|h\|_{L^\nf}=1$.
In summary, it remains to verify \eqref{e90} with $f_1$ replaced by $f_{1,1}$, which means that we can assume \eqref{e004}.

Next we show why \eqref{e5} is sufficient.
By a change of variables, we have
\begin{align}
    &\|T(f_1,\, f_2)\|_{L^1}\nonumber\\
    =&\int_{\bbr^2}|\int_{\bbr^2}f_1(x)f_2(x+P(t)-B_R t)r(t)dt|dx\nonumber\\
    \le&\|f_1\|_{L^{\frac{3}{2}}(\bbr^2)}\Big\|\int_{\bbr^2}f_2(x+P(t)-B_R t)r(t)dt\Big\|_{L^3(\bbr^2)},\label{e6}
\end{align}
where $B_Rt=(\frac{t_1}{R},\frac{t_2}{R^2})$.
Let 
$$
L(f)(x)=\int_{\bbr^2}f(x+P(t)-B_R t)r(t)dt.
$$
Applying Minkowski's inequality, we have
\begin{equation}\label{e002}
   \|L(f)\|_{L^2(\bbr^2)}\lesssim \|f\|_{L^2(\bbr^2)}. 
\end{equation}

We take a change of variables with
$$
s=(s_1,s_2)=P(t)-B_R t.
$$
The corresponding Jacobian  is 
$$
|\f{\partial s}{\partial t}|=\big|6t_1t_2^2-6t_1^2t_2-\frac{3t_2^2}{R}-2\frac{t_1}{R^2}+\frac{1}{R^3}\big|,
$$
which is greater than $A^{-1}$ by assumption \eqref{e91}. As a result,
\begin{equation}
	|Lf(x)|\lesssim A\int_{\bbr^2}|f(x+s)|ds,,
\end{equation}
which implies that 
\begin{equation}\label{e003}
    \|L(f)\|_{L^\nf(\bbr^2)}\lesssim A \|f\|_{L^1(\bbr^2)}.
\end{equation}

We can interpolate between \eqref{e002} and \eqref{e003} to obtain that 
$$
\|L(f)\|_{L^3(\bbr^2)}\lesssim A^{1/3}\|f\|_{L^{3/2}(\bbr^2)}.
$$
As a result, we conclude from this and  \eqref{e6} to obtain 
$$
\|T(f_1,\, f_2)\|_{L^1}\lesssim A^{O(1)}\|f_1\|_{L^{3/2}}\|f_2\|_{L^{3/2}},
$$
which combined with \eqref{e5} implies  \eqref{e90} with the additional condition \eqref{e004} by interpolation.

\medskip

\end{proof}

To proceed, we need a decomposition from the frequency side, which was essentially proved in \cite[Lemma 3.2]{Christ2020c}.

\begin{lm}\label{t18}
    Let $\tilde N\ge 1$, $\rho\in(0,\,1)$ and $f\in L^2(\bbr^2)$. There exists a decomposition $f=f_\flat+f_\sharp$ satisfying the following properties.

    (i) One has 
    $$\|f_\sharp\|_{L^2}+\|f_\flat\|_{L^2}\lesssim\|f\|_{L^2}.$$

    (ii) The function $f_\sharp$ is defined by 
    $$
     f_\sharp(x)=\sum_{n=1}^{N_0}h_n(Rx_1,\,R^2x_2)e^{i\alpha_n\cdot x}
    $$
    where $h_n$ is a smooth function satisfying $\|\partial^\beta h_n\|_\infty\lesssim_{\beta}\tilde N^{|\beta|}\|f\|_\infty$ for $\beta\in\bbn^2$, $\|h_n\|_{L^2}\lesssim R^{3/2}\|f\|_{L^2}$ and $\spt \widehat{h_n}\sset [-\tilde N,\,\tilde N]^2$, $N_0\lesssim\rho^{-1}$, and  $\alpha_n\in\bbr^2$. Moreover,  $\alpha_n\in\spt \widehat{f}+[-10R\tilde N,10R\tilde N]\times[-10R^2\tilde N,10R^2\tilde N]$.

    (iii) $f_\flat$ satisfies the estimate
    $$
    \int_{\bbr^2}\int_{|(\xi_1/R,\,\xi_2/R^2)|\le \tilde N}|\widehat{D_sf_\flat}|^2(\xi)d\xi ds\lesssim\rho \|f\|^4_{L^2},
    $$
    where  
    $$
    D_{s}f(x)=f(x+s)\overline{f(x)}.
    $$
\end{lm}

\begin{rmk}
    When $R = 1$, Lemma~\ref{t15} reduces to the two-dimensional version of \cite[Lemma 3.2]{Christ2020c}, which can be obtained directly by a similar argument. Roughly speaking, in the case $R=1$,  letting
       $$
    J=\{ n\in \bbz^2 :\,\int_{B(\tilde Nn,10\tilde N)}|\widehat{f}(\xi)|^2d\xi\ge\rho\|f\|_2^2\},
    $$
    we define
    \begin{align*}
        &\widehat{f_\sharp}(\xi)=\sum_{n\in J}\varphi(\tilde N^{-1}(\xi+n))\widehat{f}(\xi),\\
        &\widehat{f_\flat}(\xi)=\sum_{n\in \bbz^2\setminus J}\varphi(\tilde N^{-1}(\xi+n))\widehat{f}(\xi)
    \end{align*}
    for an appropriate smooth bump function $\varphi$.
    The case $R > 1$ follows from the case $R = 1$ by a simple change of variables $x_1\mapsto x_1/R$, $x_2\mapsto x_2/R^2$.
\end{rmk}

Because of the linearity of $T$, we can suppose that $\|f_1\|_{L^{\infty}}=\|f_2\|_{L^{\infty}}=1$. 

Let $\tilde\psi_1\in \mathscr{S}(\bbr^2)$ satisfy $\widehat{\tilde\psi_1}(\xi)=1, \forall\ |\xi|\le20 $ and 
$
\spt\widehat{\tilde\psi_1}\sset B(0,30).
$
Let $\tilde\psi_2\in \mathscr{S}(\bbr^2)$ satisfy 
$$
\widehat{\tilde\psi_2}(\xi)=1, \forall\ |\xi|\in[1,2] \text{ and }\spt \widehat{\tilde \psi_2}\sset B(0,3)\backslash B(0,1/2).
$$
We define $\psi_1(x)=\ \la^2R^3\tilde\psi_1(R\la x_1,\,R^2\la x_2)$ and $\psi_2(x)=\ \la^2\tilde\psi_2(\la x)$, whose $L^1$-norms are bounded uniformly in $\la$ and $R$.

By the Fourier support assumptions on $f_j$, we see that $f_j=f_j\ast\psi_j$. Let $\eta\in C^\infty_0(\bbr^2)$ be supported in $[-1,\,1]^2$, satisfying $\sum_{n\in\bbz^2}\eta(x+n)=1$ for any $x\in\bbr^2$. For $m\in\bbz^2$, let $\eta_{1,m}(x)=\eta\big((R\la^\gamma x_1,\,R^2\la^\gamma x_2\big)-m)$ and $\eta_{2,m}(x)=\eta(\la^\gamma x-m)$, where $\gamma\in(\tf12,1)$ will be determined later. 
We denote $f_{j,m}=\psi_j\ast(\eta_{j,m}f_j),$ then $f_j=\sum_{m\in\bbz^2}f_{j,m}.$ 

Fix $c>1>\delta_j>0$, and $\gamma'\in(\gamma,1)$, which are to be determined. 
We apply  Lemma \ref{t18} with 
$
    \tilde N=c\la^{\gamma'},\, \rho=\la^{-\delta_2}, R=1
$
to get $f_{2,m}=f_{2,m,\flat}+f_{2,m,\sharp}$. We apply  Lemma~\ref{t18} with 
$
    \tilde N=c\la^{\gamma'},\, \rho=\la^{-\delta_1}
$
to get $f_{1,m}=f_{1,m,\flat}+f_{1,m,\sharp}$. Moreover, we have 
\begin{align*}
    &f_{1,m,\sharp}(x)=\sum_{n=1}^{N_{1,m}}h_{1,m,n}(Rx_1,\,R^2x_2)e^{i\alpha_{1,m,n}\cdot x}\\
    &f_{2,m,\sharp}(x)=\sum_{n=1}^{N_{2,m}}h_{2,m,n}(x)e^{i\alpha_{2,m,n}\cdot x},
\end{align*}
where $\|f_{j,m,\sharp}\|_\infty \lesssim N_{j,m}\lesssim\la^{\delta_j}$, $\|\partial^\alpha h_{j,m,n}\|_\infty\lesssim\la^{|\alpha|\gamma'},$ and
\begin{align}
    \int_{\bbr^2}\int_{|B_R\xi|\le \la^{\gamma'}}|\widehat{D_sf_{1,m,\flat}}|^2(\xi)d\xi ds
    \lesssim &\la^{-\delta_1}\|f_{1,m}\|^4_{L^2}\nonumber\\
    \lesssim &\la^{-\delta_1}\|\eta_{1,m}f_j\|^4_{L^2}\nonumber\\
    \le& R^{-6}\la^{-\delta_1-4\gamma}\label{e46}\\
    \int_{\bbr^2}\int_{|\xi|\le \la^{\gamma'}}|\widehat{D_sf_{2,m,\flat}}|^2(\xi)d\xi ds\lesssim&\la^{-\delta_2}\|f_{2,m}\|^4_{L^2}\nonumber\\
    \lesssim&\la^{-\delta_2}\|\eta_{2,m}f_2\|^4_{L^2}\nonumber\\
    \le&\la^{-\delta_2-4\gamma}\label{e47}
\end{align}
where we use $\|f_j\|_{L^\nf}\lesssim 1$.

Let $\widetilde{\eta}\in C^\infty_0$ satisfy $\widetilde{\eta}(x)=0$ for any $x\notin[-\tf65,\,\tf65]^2$ and $\widetilde{\eta}(x)=1$ for any $x\in[-\tf{11}{10},\tf{11}{10}]^2$. 
We denote $\widetilde{\eta}_{1,m}(x)=\widetilde{\eta}(\la^\gamma \big(Rx_1,\,R^2x_2)-m\big)$ and $\widetilde{\eta}_{2,m}(x)=\widetilde{\eta}(\la^\gamma x-m)$, and define $f_{j,m,err}=(1-\widetilde{\eta}_{j,m})f_{j,m}$. 
By definition, for any $K\in\bbz^+$ we have $$
|\psi_1(x)|\lesssim_K\la^2R^3|\la (Rx_1,\,R^2x_2)|^{-K}\text{ and }|\psi_2(x)|\lesssim_K\la^2|\la x|^{-K}.
$$ 
So for $x\in \spt(1-\widetilde\eta_{j,m})$,
we have
$$
|\psi_j\ast(\eta_{j,m}f_j)(x)|\lesssim_\gamma\la^{-10}
$$ 
by taking $K$ large enough. 
As a result
\begin{equation}\label{e30}
       \|f_{j,m,err}\|_{L^\nf}=\|(1-\widetilde{\eta}_{j,m})(\psi_j\ast(\eta_{j,m}f_j))\|_{L^\infty}\lesssim\la^{-10}.
\end{equation}

Let
\begin{align*}
    &f_{j,\flat}=\sum_{m\in \bbz^2}\widetilde{\eta}_{j,m}f_{j,m,\flat},\\
    &f_{j,\sharp}=\sum_{m\in \bbz^2}\widetilde{\eta}_{j,m}f_{j,m,\sharp},\\
    &f_{j,err}=\sum_{m\in \bbz^2}f_{j,m,err}.
\end{align*}
As the supports of $\widetilde{\eta}_m$ are finitely overlapping, we can get 
\begin{equation}\label{e071}
\|f_{j,\flat}\|_{\infty}\lesssim\la^{\delta_j} \text{ and} \quad
\|f_{j,\sharp}\|_{\infty}\lesssim\la^{\delta_j}.
\end{equation}
Because of the Bernstein inequality and $\spt \widehat{f_{j,m,\flat}}$ and $\spt \widehat{f_{j,m,\sharp}}$, we  get that
\begin{align}
    &\|\partial^\alpha f_{1,\flat}(\cdot/R,\,\cdot/R^2)\|_\infty\lesssim_{|\alpha|} \la^{\delta_1+|\alpha|}\text{ and }\|\partial^\alpha f_{2,\flat}\|_\infty\lesssim_{|\alpha|} \la^{\delta_2+|\alpha|}\label{e70};\\
    &\|\partial^\alpha f_{1,\sharp}(\cdot/R,\,\cdot/R^2)\|_\infty\lesssim_{|\alpha|} \la^{\delta_1+|\alpha|}\text{ and }\|\partial^\alpha f_{2,\sharp}\|_\infty\lesssim_{|\alpha|} \la^{\delta_2+|\alpha|}\label{e71}.
\end{align}

Let
\begin{align*}
    &T_\sharp=T(f_{1,\sharp},\,f_{2,\sharp}),\\
    &T_\flat=T(f_{1},\,f_{2,\flat})+T(f_{1,\flat},\,f_{2,\sharp}),\\
    &T_{err}=T(f_{1,err},\,f_{2,\sharp})+T(f_{1},\,f_{2,err}).
\end{align*}
We  find $T(f_1,\,f_2)=T_{\sharp}+T_{\flat}+T_{err}.$

Concerning the error term, we obtain from \eqref{e30} and \eqref{e071} the estimate 
\begin{equation}\label{e31}
\|T_{err}\|_{L^1([0,1]^2)}\lesssim\|f_{1}\|_{L^\infty}\|f_{2,err}\|_{L^\infty}+\|f_{1,err}\|_{L^\infty}\|f_{2,\sharp}\|_{L^\infty}\lesssim\la^{-10+\delta_1+\delta_2}.
\end{equation}

We estimate $T_{\sharp}$ and $T_{\flat}$ in following two sections respectively.

\section{Estimate for $\|T_\flat\|_{L^1}$}\label{s4'}
\subsection{Estimate for supports}
In this section, we estimate the $L^1([0,1]^2)$-norm of 
$$
T_\flat=T(f_{1},\,f_{2,\flat})+T(f_{1,\flat},\,f_{2,\sharp}).
$$ 
We will focus on estimating $\|T(f_{1},\,f_{2,\flat})\|_{L^1([0,1]^2)}$, as the argument for the term $\|T(f_{1,\flat},\,f_{2,\,\sharp})\|_{L^1([0,1]^2)}$ is similar
by  normalizing $\|f_{2,\sharp}\|_\infty$ appropriately. We remark that in handling the second term, we require $\de_1$ to be sufficiently small. 

By definition, 
$$
f_{2,\flat}=\sum_{m\in \bbz^2}\widetilde{\eta}_mf_{2,m,\flat}=\la^{\de_2}\sum_{m\in\bbz^2}g_{2,m},
$$
where
$$
g_{2,m}:=\la^{-\de_2}\widetilde{\eta}_{2,m}f_{2,m,\flat}
$$
satisfies  $\spt g_{2,m}\subseteq\spt\tilde\eta_{2,m}$ and $\|g_{2,m}\|_\infty\lesssim1$.
We define $g_{1,m}=\eta_{1,m}f_1$, then 
$$
f_1=\sum_{m\in\bbz^2} g_{1,m}.
$$
Let 
$$
J_j=\{m_j\in\bbz^2:\,\spt\widetilde\eta_{j,m_j}\bigcap[0,\,2]^2\neq\emptyset\},
$$
and $J=J_1\times J_2$, whose cardinality is $O(R^3\la^{4\gamma}).$ By the support of $ r$, 
$$
\|T(g_{1,m_1},g_{2,m_2})\|_{L^1(\zozo)}\neq0
$$
only if
$(m_1,\,m_2)\in J.$ In particular,
$$
\la^{-\delta_2}\|T(f_1,\, f_{2,\flat})\|_{L^1(\zozo)}\le \sum_{m\in J}\|T(g_{1,m_1},\, g_{2,m_2})\|_{L^1([0,1]^2)}.
$$

We observe that the support of $g_{1,m_1}(x+B_Rt)g_{2,m_2}(x+P(t))r(t)$ is contained in 

$$
\{(x,t)\in \bbr^2\times \bbr^2:\ t\in\spt r,\ x+B_Rt\in \la^{-\gamma}B_Rm_1+O(\la^{-\gamma})\},
$$
which however can be reduced with the help of $P(t)$.

For $(x,t)\in\bbr^2\times \bbr^2$ such that $g_{1,m_1}(x+B_Rt)g_{2,m_2}(x+P(t))r(t)\neq 0$ we have
\begin{equation}
\begin{cases}
	x+B_Rt\in \big(\la^{-\gamma}R^{-1}m_{1,1}+O(R^{-1}\la^{-\gamma}),\,\la^{-\gamma}R^{-2}m_{1,2}+O(R^{-2}\la^{-\gamma})\big)\\
	x+P(t)\in  \la^{-\gamma}m_{2}+O(\lambda^{-\gamma}).
\end{cases}\label{e10}
\end{equation}
We denote
$$
A_m=\{(x,\,t):\ (x,t)\text{ satisfies \eqref{e10} and }t\in\spt r\}.
    $$
    
\begin{lm}\label{t30}
When $m=(m_1,m_2)\in J$ is fixed,
$A_m$ is contained the union of at most  six sets, each of which is  a rectangular box of dimensions $\sim AR^{-1}\la^{-\gamma}\times AR^{-2}\la^{-\gamma} \times  A\la^{-\gamma} \times  A\la^{-\gamma} $. 
In particular,
\begin{equation}
	\|T(g_{1,m_1},\, g_{2,m_2})\|_{L^1([0,1]^2)}\lesssim AR^{-\frac{3}{2}}\lambda^{-\gamma}\|T(g_{1,m_1},\, g_{2,m_2})\|_{L^2([0,1]^2)}.\label{e12}
\end{equation} 
\end{lm}

To prove this result, we need the following simple fact of matrix, whose proof is  given for the sake of completeness.
\begin{lm}\label{t5}
	Suppose that $B=(a_{pq})_{n\times n}$ is an $n\times n$ matrix. Let 
    $$
    \tau_0(B)=\min\{|\tau|:\tau \text{ is an eigenvalue of } B\}.
    $$
    Suppose that $\|B\|_2\le 1$, where $\|B\|_2=(\sum_{p,q=1}^{n}|a_{pq}|^2)^{1/2}$, and $|\det(B)|
    \ge c_1$, then $|\tau_0(B)|\ge c_1$.
    
\end{lm}
\begin{proof}
    	Let $\rho(B)=\max\{|\tau|:\tau\text{ is an eigenvalue of }B\}$ be the spectral radius of $B$,
        which is bounded by $\|B\|_2\le 1$ by inspecting an eigenvector corresponding to $\rho(B)$. 
       Then we have
        $$
        c_1\le |\det(B)|\le\rho(B)^{n-1}\tau_0(B)\le \tau_0(B).
        $$
        We finish the proof.
\end{proof}

\begin{proof}[Proof of Lemma~\ref{t30}]

Defining a mapping $F$: $(x_1,\, x_2,\, t_1,\,t_2)\to(z_1,\,z_2,\,z_3,\,z_4)$ by
\begin{equation}
	\begin{cases}
		x+B_Rt=(z_1,\,z_2)\\
		x+P(t)=(z_3,\,z_4),\\
	\end{cases}\label{e11}
\end{equation}
we obtain
$$
(z_3-t_1^2-z_1+t_1/R)(z_3-t_1^2-z_1+t_1/R-1/R^2)^2=(z_4-z_2-t_1^3)^2
$$ 
by eliminating $x$ and $t_2$.
When $z_1,\,z_2,\,z_3,\,z_4$ are fixed, as an equation of $t$, it is an equation of degree $6$, so there are at most six solutions. In particular, there are at most six tuples  $(x,\, y,\, t_1,\,t_2)$ satisfying \eqref{e11}. So we can define six  mappings $G_i$, $1\le i\le 6$, from $(z_1,\,z_2,\,z_3,\,z_4)$ to $(x,\, y,\, t_1,\,t_2)$ describing all solutions to \eqref{e11}. We fix $i$ and denote $G_i$ by $G$ for simplicity.

We claim that, when $m$ is fixed and  $G_i(z)\in A_m$, $G_i(z)$ lies in a rectangular box of dimensions $\sim AR^{-1}\la^{-\gamma}\times AR^{-2}\la^{-\gamma} \times  A\la^{-\gamma} \times  A\la^{-\gamma} $ for any $1\le i\le 6$,  where $z=(z_1,z_2,z_3,z_4)$. To prove this, it suffices to estimate $G(z)-G(z')$.
We observe first that the Jacobian of $F$ is
$$J(F)=
\begin{bmatrix}
	1&0&1/R&0\\0&1&0&1/R^2\\1&0&2t_1&2t_2\\0&1&3t_1^2&3t_2^2
\end{bmatrix}, $$
whose determinant is
$|det(J(F))|=\big|6t_1t_2^2-6t_1^2t_2-\frac{3t_2^2}{R}-2\frac{t_1}{R^2}+\frac{1}{R^3}\big|$. By \eqref{e91}, we  get  $|det(J(F)|\ge A^{-1}$, which implies that $\tau_0(J(F))\gtrsim A^{-1}$ by Lemma \ref{t5}.
By the inverse function theorem, we can get $J(G_i)=(J(F))^{-1}$, which 
combined with the mean value theorem yields that
$$
|G_i(z)-G_i(z')|\lesssim \tau_0(J(F))^{-1}|z-z'|\lesssim A|z-z'|.
$$
As $|z-z'|\lesssim \la^{-\gamma}$ by \eqref{e10},
we obtain further that $G_i(z)$ lies in a ball with radius $O(A\la^{-\gamma})$. In particular, 
$x$ lies in a rectangle of dimensions  $O(AR^{-1}\la^{-\gamma})\times O(AR^{-2}\la^{-\gamma})$ recalling the first equation in \eqref{e10}. We finish the proof of the claim and the lemma.

\end{proof}

Without loss of generality, we may fix one rectangular box.  In particular $supp(T(g_{1,m_1},\, g_{2,m_2})(x))$ is contained in a rectangle of dimensions $\sim AR^{-1}\la^{-\gamma}\times AR^{-2}\la^{-\gamma}$.

\subsection{$T T^*$}
We write 
\begin{align}
	&\|T(g_{1,m_1},\, g_{2,m_2})\|_2^2\\
    =&\int_{\bbr^6}g_{1,m_1}\big(x+B_R(t+s))\big)
	\overline{g_{1,m_1}}
		\big(x+B_Rt\big)g_{2,m_2}(
x+P(t+s))
\nonumber\\&\overline{g_{2,m_2}}(
	x+P(t)
)r(t)r(t+s)dxdtds.\label{e13}
\end{align}
Because 
$	x+B_Rt$, 
$x+B_R(t+s)\in \spt g_{1,m_1}$, 
we  find $|s|\lesssim \lambda^{-\gamma}$, where the implicit constant is independent of $R$. 
Let $(\overline{x},\,\overline{t}_m)\in A_{m}$, then $|t-\overline{t}_m|\lesssim A\lambda^{-\gamma}$ for any $(x,\,t)\in A_m$.  To simplify the notation, we write $\overline{t}_m=(\overline t_{m,1}, \overline t_{m,2})$ as $\overline{t}=(\overline t_{1}, \overline t_{2})$ when $m$ is fixed.

As $P(t+s)-P(t)$ is nonlinear in $s$, we may overcome this obstacle by fixing $t$ to get a linear approximation.
From the definition of $g_{2,m_2}$ and \eqref{e70}, we get $\|\nabla g_{2,m_2}\|_{\infty}\lesssim\lambda$. 
By the mean value theorem, we can get 
\begin{align}
    &g_{2,m_2}(x+P(t+s))\nonumber\\
   = &g_{2,m_2}(x+P(t)+M(\overline{t})s)+O(A^2\lambda^{-2\gamma+1})\label{e14},
\end{align}
where $$
M(\overline{t})=
\begin{bmatrix}
    2\overline{t_1}&2\overline{t_2}\\
    3\overline{t_1}^2&3\overline{t_2}^2
\end{bmatrix}
$$
and 
$$
M(\overline{t})s=(2s_1\overline{t_1}+2s_2\overline{t_2},\,3s_1\overline{t_1}^2+3s_2\overline{t_2}^2).
$$
Recalling that $$D_{s}f(x)=f(x+s)\overline{f(x)},
$$
{\color{red}we obtain from  \eqref{e13}  and \eqref{e14} that
\begin{align*}
   &\|T(g_{1,m_1},\, g_{2,m_2})\|_2^2\\
   \lesssim&\int_{|s|\lesssim \lambda^{-\gamma}}\big|\int_{\bbr^4}D_{B_Rs}g_{1,m_1}\big(x+B_Rt)D_{M(\overline{t})s}g_{2,m_2}(x+P(t))\\
   +&O(A^2\la^{-2\delta}))r(t)r(t+s)dxdt\big|ds,
\end{align*}
where $\delta=\gamma-\frac{1}{2}>0.$ 
}

Plugging this estimate into \eqref{e12}, we finally obtain
\begin{align}
	&\|T(f_1,\,f_{2,\flat})\|_{L^1(\zozo)}\nonumber\\
	\lesssim&AR^{-3/2}\lambda^{-\gamma+\de_2}\sum_{m\in J}\Big(\int_{|s|\lesssim \lambda^{-\gamma}}\big|\int_{\bbr^4}D_{B_Rs}g_{1,m_1}\big(x+B_Rt)D_{M(\overline{t})s}g_{2,m_2}(x+P(t)))\notag\\
	&r(t)r(t+s)dxdt\big|ds\Big)^{\frac{1}{2}}\label{e15}\\
    &+O(A^{O(1)}\la^{-\delta+\de_2})\label{e088},
\end{align}
where we use Lemma~\ref{t30} to control the error term \eqref{e088}.
 By the Cauchy-Schwarz inequality, the main term  \eqref{e15} can be controlled by
\begin{align}
	&A\lambda^{\gamma+\de_2}\nonumber\\
    &\Big(\int_{
				 |s|\lesssim \lambda^{-\gamma}
		}\sum_{m\in J}\big|\int_{\bbr^4}D_{B_Rs}g_{1,m_1}\big(x+B_Rt)D_{M(\overline{t})s}g_{2,m_2}(x+P(t)))r(t)r(t+s)dxdt\big|ds\Big)^{\frac{1}{2}}.\label{e16}
\end{align}

In \cite{Christ2020c}, a similar term is handled by expanding the inner functions in Fourier series. We  instead work directly with the Fourier transform, as Lemma~\ref{t18}  (iii) is formulated using Fourier transform. This approach is slightly simpler, although both approaches are essentially equivalent.

\subsection{High frequency}
 
Recall that the support of $D_{B_Rs}g_{1,m_1}$ is contained in a rectangle of dimensions $\sim R^{-1}\lambda^{-\gamma}\times R^{-2}\lambda^{-\gamma}$ and the support of $D_{M(\overline{t})s}g_{2,m_2}$ is contained in a square with side length $\lambda^{-\gamma}$ .
Let $\tilde{\tilde\eta}\in C^\infty(\bbr)$ satisfy that, for $x\notin[-1,\,2]^2$, $\tilde{\tilde\eta}(x)=0$, and, for $x\in\zozo$, $\tilde{\tilde\eta}(x)=1$.
From the support of $D_sg_{j,m_j}$, we can get that
\begin{align}
    &\int_{\bbr^4}D_{B_Rs}g_{1,m_1}(x+B_Rt)D_{M(\overline{t})s}g_{2,m_2}(x+P(t))
r(t)r(t+s)dxdt\nonumber\\
=&\int_{\bbr^4}D_{B_Rs}g_{1,m_1}(x+B_Rt)D_{M(\overline{t})s}g_{2,m_2}(x+P(t))
\zeta_m(x,\,t)dxdt,\nonumber\\
=&\int_{\bbr^4}\widehat{D_{B_Rs}g}_{1,m_1}(\xi)\widehat{D_{M(\overline{t})s}g}_{2,m_2}(\eta)\nonumber\\
    &\qquad \int_{\bbr^4}e^{2\pi i((x+B_Rt)\cdot\xi+(x+P(t))\cdot\eta)}\zeta_m(x,\,t)dxdtd\eta d\xi.\label{e96}
\end{align}
where
$$
\zeta_m(x,\,t)=\tilde{\tilde\eta}((\la^{\gamma}((Rx_1,R^2x_2)+t)-m_1)\tilde{\tilde\eta}(\la^{\gamma}(x+P(t))-m_2)r(t)r(t+s)
$$

The contribution of the high frequency part, namely $(B_R\xi,\eta)\notin B(0,\la^{1+\delta_3})$, is small.

\begin{prop}\label{t50}
For any $K\ge1$, we  have
    \begin{align*}
    &\int_{				 |s|\lesssim \lambda^{-\gamma}		}
     \sum_{m\in J}\big|\int_{\{(\xi,\eta):|B_R\xi|\le\la^{1+\delta_3},|\eta|\le\la^{1+\delta_3}\}^c}\widehat{D_{B_Rs}g}_{1,m_1}(\xi)\widehat{D_{M(\overline{t})s}g}_{2,m_2}(\eta)\\
    &\qquad \int_{\bbr^4}e^{2\pi i((x+B_Rt)\cdot\xi+(x+P(t))\cdot\eta)}\zeta_m(x,\,t)dxdtd\eta d\xi\big| ds\\
    &    \lesssim_K A^4\lambda^{-K}.
    \end{align*}
\end{prop}

Let
\begin{align*}
    &\wh {G_1^H}(\xi)=\wh{D_{B_Rs}g_{1,m_1}}(\xi) \chi_{(B(0,\la^{1+\delta_3}))^c}(B_R\xi),\\
    &\wh {G_2^H}(\eta)=\wh{D_{M(\overline t)s}g_{2,m_2}}(\eta) \chi_{(B(0,\la^{1+\delta_3}))^c}(\eta),
\end{align*}
and
\begin{align*}
    &\wh {G_1^L}(\xi)=\wh{D_{B_Rs}g_{1,m_1}}(\xi) \chi_{B(0,\la^{1+\delta_3})}(B_R\xi),\\
    &\wh {G_2^L}(\eta)=\wh{D_{M(\overline t)s}g_{2,m_2}}(\eta) \chi_{B(0,\la^{1+\delta_3})}(\eta).
\end{align*}

\begin{lm}\label{t10001}
For $\la\ge 1$ and $R\ge 1$, we have
$$
|\int_{\bbr^4} G_1^H(x+B_Rt) D_{M(\overline t)s}g_{2,m_2}(x+P(t))\zeta_m(x,t)dxdt|\lesssim_K A^4 R^{-3}\la^{-K}
$$
and
$$
|\int_{\bbr^4} G_1^L(x+B_Rt) G_2^H(x+P(t))\zeta_m(x,t)dxdt|\lesssim_K A^4 R^{-3}\la^{-K}.
$$
\end{lm}

\begin{proof}
Since $\|\partial^\alpha D_{B_Rs}g_{1,m_1}\|_\infty\lesssim R^{\alpha_1+2\alpha_2}\la^{|\alpha|}$ and 
$|\spt D_{B_Rs}g_{1,m_1}|\lesssim R^{-3}\la^{-2\gamma}$, we have, 
for $|B_R\xi|\ge{\la^{1+\delta_3}}$, 
$$
\big|\wh {G_1^H}(\xi)\big|=\big|\widehat{D_{B_Rs}g_{1,m_1}}(\xi)\big|\lesssim_KR^{-3}\la^{-2\gamma}(\la^{-1}|B_R\xi|)^{-K}.
$$
Similarly, because $\|\partial^\alpha D_{M(\bar t)s}g_{2,m_2}\|_\infty\lesssim \la^{|\alpha|}$, for $|\xi|\ge{\la^{1+\delta_3}}$, we have
$$
\big|\wh {G_2^H}(\xi)\big|=\big|\widehat{D_{M(\bar t)s}g_{2,m_2}}(\xi)\big|\lesssim_K\la^{-2\gamma}(\la^{-1}|\xi|)^{-K}.
$$
So 
$$
\|G_j^H\|_\infty\le\|\widehat{G_j^H}\|_1\lesssim_{K,\delta_3}\la^{-K}.
$$
Notice that
$$
\|G_j^H+G_j^L\|_\infty\le1,
$$
so we also have $\|G^{L}_{j}\|_\infty\lesssim1$.
Combining the above estimate, we have
\begin{align*}
    &|\int_{\bbr^4} G_1^H(x+B_Rt) D_{M(\overline t)s}g_{2,m_2}(x+P(t))\zeta_m(x,t)dxdt|\\
    \le &|A_m|\|G_1^H\|_\infty\|D_{M(\overline t)s}g_{2,m_2}\|_\infty\\
    \lesssim_K &A^4 R^{-3}\la^{-K}.
\end{align*}
Similarly, we can also get 
$$
|\int_{\bbr^4} G_1^L(x+B_Rt) G_2^H(x+P(t))\zeta_m(x,t)dxdt|\lesssim_K A^4 R^{-3}\la^{-K}
$$
\end{proof}
\begin{proof}[Proof of Proposition~\ref{t50}]
    By  Lemma~\ref{t10001} and the Fourier inversion formula, we have
    \begin{align*}
    &\int_{				 |s|\lesssim \lambda^{-\gamma}		}
     \sum_{m\in J}\big|\int_{\{(\xi,\eta):|B_R\xi|\le\la^{1+\delta_3},|\eta|\le\la^{1+\delta_3}\}^c}\widehat{D_{B_Rs}g}_{1,m_1}(\xi)\widehat{D_{M(\overline{t})s}g}_{2,m_2}(\eta)\\
    &\qquad \int_{\bbr^4}e^{2\pi i((x+B_Rt)\cdot\xi+(x+P(t))\cdot\eta)}\zeta_m(x,\,t)dxdtd\eta d\xi\big| ds\\
    =&\int_{				 |s|\lesssim \lambda^{-\gamma}		}
     \sum_{m\in J}|\int_{\bbr^4} G_1^H(x+B_Rt) D_{M(\overline t)s}g_{2,m_2}(x+P(t))\zeta_m(x,t)dxdt|ds\\
    &+
\int_{				 |s|\lesssim \lambda^{-\gamma}		}
     \sum_{m\in J}|\int_{\bbr^4} G_1^L(x+B_Rt) G_2^H(x+P(t))\zeta_m(x,t)dxdt|ds\\
     \lesssim& A^4 \la^{-K},
    \end{align*}
    which implies Proposition~\ref{t50}.
\end{proof}
By Proposition~\ref{t50} and \eqref{e96}, to estimate \eqref{e16}, we only need to estimate

\begin{align}
  &A\lambda^{\gamma+\de_2}\Big(\int_{
				 |s|\lesssim \lambda^{-\gamma}
		}\sum_{m\in J}\big|\int_{|B_R\xi|\le\la^{1+\delta_3}}\int_{|\eta|\le\la^{1+\delta_3}}\widehat{D_{B_Rs}g}_{1,m_1}(\xi)\widehat{D_{M(\overline{t})s}g}_{2,m_2}(\eta)\label{e205}\\
    &\qquad \int_{\bbr^4}e^{2\pi i((x+B_Rt)\cdot\xi+(x+P(t))\cdot\eta)}\zeta_m(x,\,t)dxdtd\eta d\xi\big|ds)^{\frac12}.\notag
\end{align}

\subsection{Low frequency}

Fix $m\in J$ and $
    |s|\lesssim \lambda^{-\gamma}
$.
Fix $\delta_4>0$ small enough. Let 
\begin{align*}
     S_m:=&\{(\xi,\,\eta)\in\bbr^4:\,|B_R\xi|\le \lambda^{1+\delta_3},\,|\eta|\le\lambda^{1+\delta_3},\\
     &        |\xi_1+\eta_1|\le R\la^{\gamma+\delta_4},\,|\xi_2+\eta_2|\le R^2\la^{\gamma+\delta_4}\\
        &|\xi_1/R+2\overline{t_1}\eta_{1}+3\overline{t_1}^2\eta_{2}|\le \la^{\gamma+\delta_4},\,
        |\xi_{2}/R^2+2\overline{t_2}\eta_{1}+3\overline{t_2}^2\eta_{2}|\le \la^{\gamma+\delta_4}\}.
\end{align*}
This is a set depending on $m$ as $\overline t$ relies on $m$.
We will show that, to control \eqref{e205}, we only need to consider  $(\xi,\,\eta)\in S_m$. 

\begin{prop}\label{t35}
    
    For $\la\ge 1$ and $R\ge 1$, we have
    \begin{align*}
    &\int_{|B_R\xi|\le\la^{1+\delta_3}}\int_{|\eta|\le\la^{1+\delta_3}}\big|\widehat{D_{B_Rs}g}_{1,m_1}(\xi)\big|\big|\widehat{D_{M(\overline{t})s}g}_{2,m_2}(\eta)\ \big|\chi_{(S_m)^c}(\xi,\,\eta)\\&
    \Big|\int_{\bbr^4}e^{2\pi i((x+B_Rt)\cdot\xi+(x+P(t))\cdot\eta)}\zeta_m(x,\,t)dxdt\Big|d\eta d\xi\\
    \lesssim_{\delta_3,\delta_4,N}& R^{-3}A^{O(1)}\la^{-N}.
\end{align*}
\end{prop}

\begin{proof}

 Actually, by integration by parts, we can show that for 
 $$
 \{(\xi,\,\eta)\in\bbr^4:\,|B_R\xi|\le \lambda^{1+\delta_3},\,|\eta|\le\lambda^{1+\delta_3}\}\setminus S_m,
 $$
 we have
\begin{equation}\label{e074}
     \Big|\int_{\bbr^4}e^{2\pi i((x+B_Rt)\cdot\xi+(x+P(t))\cdot\eta)}\zeta_m(x,\,t)dxdt\Big|\lesssim_{K,\delta_4}A^{O_{K,\delta_4}(1)}R^{-3}\la^{-K}.
\end{equation}

By the size of the support of $D_sg_{j,m_j}$ and the fact that $\|D_sg_{j,m_j}\|_\infty\le1$, we have

$$
\|\widehat{D_{B_Rs}g}_{1,m_1}\|_\infty\le\|D_{B_Rs}g_{1,m_1}\|_1\lesssim R^{-3}\la^{-2\gamma}
$$
and
$$
\|\widehat{D_{M(\overline{t})s}g}_{2,m_2}\|_\infty\le\|D_{M(\overline{t})s}g_{2,m_2}\|_1\le\la^{-2\gamma}.
$$
Combing this with \eqref{e074}, we have
\begin{align*}
    &\int_{|B_R\xi|\le\la^{1+\delta_3}}\int_{|\eta|\le\la^{1+\delta_3}}\big|\widehat{D_{B_Rs}g}_{1,m_1}(\xi)\big|\big|\widehat{D_{M(\overline{t})s}g}_{2,m_2}(\eta)\ \big|d\xi d\eta\\
    \lesssim&R^3\la^{4+4\delta_3}\|\widehat{D_{B_Rs}g}_{1,m_1}\|_\infty\|\widehat{D_{M(\overline{t})s}g}_{2,m_2}\|_\infty\\
    \le&\la^{4-4\gamma+4\delta_3}.
\end{align*}
So we complete the proof.
\end{proof}

\medskip

 To estimate \eqref{e205}, by Proposition~\ref{t35}, it remains to control
\begin{align*}
  &A\lambda^{\gamma+\de_2}\Big(\int_{
				 |s|\lesssim \lambda^{-\gamma}
		}\sum_{m\in J}\big|\int_{(\xi,\eta)\in S_m}\widehat{D_{B_Rs}g}_{1,m_1}(\xi)\widehat{D_{M(\overline{t})s}g}_{2,m_2}(\eta)\\
    &\qquad \int_{\bbr^4}e^{2\pi i((x+B_Rt)\cdot\xi+(x+P(t))\cdot\eta)}\zeta_m(x,\,t)dxdtd\eta d\xi\big|ds\Big)^{\frac12},
\end{align*}
which is bounded by
\begin{equation}\label{e29}
    A^3R^{-3/2}\la^{-\gamma+\delta_2}\Big(\int_{
				 |s|\lesssim \lambda^{-\gamma}
		}\sum_{m\in J}\int_{(\xi,\eta)\in S_m}\big|\widehat{D_{B_Rs}g}_{1,m_1}(\xi)\big|\big|\widehat{D_{M(\overline{t})s}g}_{2,m_2}(\eta)\big|d\eta d\xi ds\Big)^{\frac12},
\end{equation}
where we use that  $|\spt \zeta_m|\lesssim O(A^4R^{-3}\la^{-4\gamma})$,

Let
$$
S_m'= S_m\bigcap \{(\xi,\,\eta):\ |\eta|\le \la^{\gamma+\kappa}/2\},
$$
where $\kappa\ge 100\de_4$ is a sufficiently small constant.
We remark that, 
for $(\xi,\,\eta)\in S_m'$, 
we have 
\begin{equation}\label{e090}
|B_R\xi|\le\frac{\la^{\gamma+\kappa}}{2R}+\la^{\gamma+\delta_4}\le\la^{\la+\kappa}
\end{equation}
for $\la$ large enough.
$$
B_m=\{\eta:\,\exists\xi \text{ such that }(\xi,\eta)\in S_m\}\sset \bbr^2,
$$
$$
B_m'= B_m\bigcap B(0,\la^{\gamma+\kappa}/2).
$$
and 
$$
J_{\eta,m_2}=\{m_1:\,\eta\in B_{(m_1,m_2)}\}.
$$

We will consider the cases $(\xi,\eta)\in S_m\setminus S_m'$ and $(\xi,\eta)\in S_m'$ separately.

\begin{prop}\label{t66}
For $\la\ge 1$, and $R\ge 1$, we have
\begin{align}
        &A^3R^{-3/2}\la^{-\gamma+\delta_2}\Big(\int_{
				 |s|\lesssim \lambda^{-\gamma}
		}\sum_{m\in J}\int_{S_m\setminus S_m'}\big|\widehat{D_{B_Rs}g}_{1,m_1}(\xi)\big|\big|\widehat{D_{M(\overline{t})s}g}_{2,m_2}(\eta)\big|d\eta d\xi ds\Big)^{\frac12}\nonumber\\
        \lesssim & A^{O(1)}\lambda^{-\frac{\kappa-9\delta_4}{8}+\delta_2},\label{e214}
    \end{align}
    where the implicit constant is independent of $R$.

\end{prop}

To obtain an effective control in this case, we need to bound the number of $m_1$ involved when $\eta$ and $m_2$ are fixed. 

\begin{lm}
    Fix $\eta$ with $|\eta|\ge \la^{\gamma+\kappa}/2$ and $m_2\in J_2$. 
Then
\begin{equation}\label{e076}
    |J_{\eta,m_2}|\lesssim
    R^3+R^3\la^{2\gamma-\frac{\kappa-\delta_4}{2}}.
\end{equation}
\end{lm}
\begin{proof}
For $(\xi,\,\eta)\in S_m$, we have
\begin{equation}
          |\xi_1+\eta_1|\le R\la^{\gamma+\delta_4},\,|\xi_2+\eta_2|\le R^2\la^{\gamma+\delta_4}\label{e210}
  \end{equation}      
  and 
        \begin{equation}
        |\xi_1/R+2\overline{t_1}\eta_{1}+3\overline{t_1}^2\eta_{2}|\le \la^{\gamma+\delta_4},\,
        |\xi_{2}/R^2+2\overline{t_2}\eta_{1}+3\overline{t_2}^2\eta_{2}|\le \la^{\gamma+\delta_4}.\label{e211}
\end{equation}
Therefore, for $\eta\in B_m$, we have
$$
|-\eta_1/R+2\overline{t_1}\eta_{1}+3\overline{t_1}^2\eta_{2}|\lesssim\la^{\gamma+\delta_4},
$$
and
$$
|-\eta_{2}/R^2+2\overline{t_2}\eta_{1}+3\overline{t_2}^2\eta_{2}|\lesssim \la^{\gamma+\delta_4}.
$$
For $\eta_2\neq0$, the set
\begin{equation}\label{e208}
    \{\overline t_{m,1}:|-\eta_1/R+2\overline{t_1}\eta_{1}+3\overline{t_1}^2\eta_{2}|\lesssim \la^{\gamma+\delta_4}\}
\end{equation}
and the set 
\begin{equation}\label{e209}
    \{\overline t_{m,2}:|-\eta_{2}/R^2+2\overline{t_2}\eta_{1}+3\overline{t_2}^2\eta_{2}|\lesssim \la^{\gamma+\delta_4}\}
\end{equation}
are contained in at most two intervals of length $O(\la^{\frac{\gamma+\delta_4}{2}}|\eta_{2}|^{-\frac{1}{2}})$.  In particular, $\overline t_m$ is contained in four squares of length $O(\la^{\frac{\delta_4+\gamma}{2}}|\eta_{2}|^{-\frac{1}{2}})$. 
Recalling that $(\overline x,\,\overline t_m)\in A_m$, we  can use the supports of $g_{1,m_1}$ and $g_{2,m_2}$ to show that 
\begin{equation}\label{e081}
    P(\overline t_m)-B_R(\overline{t_m})=\lambda^{-\gamma}( m_2-B_R(m_1))+O(\la^{-\gamma}).
\end{equation}
Let $n=B_R(m_1)-m_2\in\bbr^2,$ and 
\begin{equation}\label{e084}
\Psi(t)=B_R(t)-P(t) 
\end{equation}
a function from $\bbr^2$ to $\bbr^2$. Because $\|J(\Psi)\|_\infty\lesssim1$, we can use the mean value theorem to obtain
\begin{equation}\label{e82}
    |n-n'|\lesssim\lambda^\gamma|\overline t_m-\overline t_{m'}|+O(1).
\end{equation}
Therefore, when $|\eta_2|\ge\frac{\la^{\gamma+\kappa}}{20}$,  we verify \eqref{e076} from the range of $\overline t_m$.

When $|\eta_{2}|\le\frac{\la^{\kappa+\gamma}}{20}$, we have $|\eta_{1}|\ge\frac{\la^{\kappa+\gamma}}{5}$, 
which implies that the set \eqref{e208} and \eqref{e209} are contained in at most two intervals of length 
$O(\la^{\delta_4+\gamma}|\eta_{1}|^{-1})$.
This implies  \eqref{e076}  for a fixed $m_2$.
 We finish the proof.

\end{proof}

\begin{proof}[Proof of Proposition~\ref{t66}]
     By \eqref{e210}, for fixed $\eta$ and $m$, we have
     \begin{equation}\label{e215}
         \big|\{\xi:\,(\xi,\eta)\in S_m\}\big|\lesssim R^3\la^{2\gamma+2\delta_4}
     \end{equation}

     For fixed $\xi$ and $ m$, because 
$$
\Big|\det\begin{bmatrix}
    2\overline t_1 &3\overline t_1^2\\
    2\overline t_2 &3\overline t_2^2
\end{bmatrix}\Big|=6|\overline t_1\overline t_2(\overline t_1-\overline t_2)|\gtrsim A^{-3},
$$
we obtain from \eqref{e211} that
\begin{equation}\label{e216}
    \big|\{\eta:\,(\xi,\eta)\in S_m\}\big|\lesssim A^{O(1)}\la^{2\gamma+2\delta_4}.
\end{equation}
These estimates combined with \eqref{e076} and the H\"older inequality imply that, for a fixed $|s|\lesssim \lambda^{-\gamma}$, 
\begin{align*}
    &\sum_{m\in J}\int_{ S_m\setminus S_m'}\big|\widehat{D_{B_Rs}g}_{1,m_1}(\xi)\big|\big|\widehat{D_{M(\overline{t})s}g}_{2,m_2}(\eta)\big|d\eta d\xi\\
    \le&\sum_{m\in J}\Big(\int_{ S_m\setminus S_m'}\big|\widehat{D_{B_Rs}g}_{1,m_1}(\xi)\big|^2 d\eta d\xi\Big)^{\frac12}\\
    &\Big(\int_{(\xi,\eta)\in S_m\setminus S_m'}\big|\widehat{D_{M(\overline{t})s}g}_{2,m_2}(\eta)\big|^2d\eta d\xi\Big)^{\frac{1}{2}}\\
    \lesssim&A^{O(1)}R^{\frac{3}{2}}\la^{2\gamma+2\delta_4}\sum_{m\in J}\|D_{B_Rs}g_{1,m_1}\|_2\Big(\int_{ B_m\setminus B_m'}\big|\widehat{D_{M(\overline{t})s}g}_{2,m_2}(\eta)\big|^2d\eta \Big)^{\frac{1}{2}}\\
    \lesssim&A^{O(1)}R^{\frac{3}{2}}\la^{3\gamma+2\delta_4}\Big(\sum_{m_2\in J_2}\int_{|\eta|\ge\frac{\la^{\gamma+\kappa}}{2}}\sum_{m_1\in J_{\eta,m_2}}\big|\widehat{D_{M(\overline{t})s}g}_{2,m_2}(\eta)\big|^2d\eta \Big)^{\frac{1}{2}}\\
    \lesssim&A^{O(1)}R^{3}
    \la^{4\gamma-\frac{\kappa-9\delta_4}{4}}
    \Big(\sum_{m_2\in J_2}\|D_{M(\overline{t})s}g_{2,m_2}\|_2^2\Big)^{\frac{1}{2}}\\
    \lesssim&A^{O(1)}R^{3}\la^{4\gamma-\frac{\kappa-9\delta_4}{4}},
\end{align*}
which implies \eqref{e214}.
\end{proof}
So to estimate \eqref{e29}, we only need to bound 
\begin{align}
    &A^3R^{-3/2}\la^{-\gamma+\delta_2}\nonumber\\
    &\Big(\int_{
				 |s|\lesssim \lambda^{-\gamma}
		}\sum_{m\in J}\int_{(\xi,\eta)\in S_m'}\big|\widehat{D_{B_Rs}g}_{1,m_1}(\xi)\big|\big|\widehat{D_{M(\overline{t})s}g}_{2,m_2}(\eta)\big|d\eta d\xi ds\big)^{\frac12}.\label{e78}
\end{align}

\begin{lm}\label{t67}
For $\la\ge 1$ and $R\ge 1$, we have
     \begin{align}
          &A^3R^{-3/2}\la^{-\gamma+\delta_2}\nonumber\\
          &\Big(\int_{
				 |s|\lesssim \lambda^{-\gamma}
		}\sum_{m\in J}\int_{(\xi,\eta)\in S_m'}\big|\widehat{D_{B_Rs}g}_{1,m_1}(\xi)\big|\big|\widehat{D_{M(\overline{t})s}g}_{2,m_2}(\eta)\big|d\eta d\xi ds\big)^{\frac12}\nonumber\\
        \lesssim &A^{O(1)}R^{\frac34}\lambda^{\frac\gamma2+\delta_4+\delta_2}\Big(\sum_{m_1\in J_1}\int_{
				 \bbr^2
		}\int_{|B_R\xi|\lesssim\la^{\gamma+\kappa}}\big|\widehat{D_{s}g}_{1,m_1}(\xi)\big|^2d\xi ds\Big)^{\frac14}\label{e222}
     \end{align}
     and 
     \begin{align}
          &A^3R^{-3/2}\la^{-\gamma+\delta_2}\nonumber\\
          &\Big(\int_{
				 |s|\lesssim \lambda^{-\gamma}
		}\sum_{m\in J}\int_{(\xi,\eta)\in S_m'}\big|\widehat{D_{B_Rs}g}_{1,m_1}(\xi)\big|\big|\widehat{D_{M(\overline{t})s}g}_{2,m_2}(\eta)\big|d\eta d\xi ds\big)^{\frac12}\nonumber\\
        \lesssim&A^{O(1)}\lambda^{\frac\gamma2+\delta_4+\delta_2}\Big(\sum_{m_2\in J_2}\int_{
				\bbr^2
		}\int_{|\eta|\lesssim\la^{\gamma+\kappa}}\big|\widehat{D_{s}g}_{2,m_2}(\eta)\big|^2d\eta ds\Big)^{\frac14}.\label{e223}
     \end{align}
\end{lm}
\begin{proof}

By \eqref{e090}, \eqref{e215}, \eqref{e216}, and the H\"older inequality, we have 
    \begin{align}
        &\int_{
				 |s|\lesssim \lambda^{-\gamma}
		}\sum_{m\in J}\int_{(\xi,\eta)\in S_m'}\big|\widehat{D_{B_Rs}g}_{1,m_1}(\xi)\big|\big|\widehat{D_{M(\overline{t})s}g}_{2,m_2}(\eta)\big|d\eta d\xi ds\nonumber\\
        \lesssim &A^{O(1)}R^{3}\lambda^{4\gamma+2\delta_4}\Big(\sum_{m_1\in J_1}\int_{
				 |s|\lesssim \lambda^{-\gamma}
		}\int_{|B_R\xi|\lesssim\la^{\lambda+\kappa}}\big|\widehat{D_{B_Rs}g}_{1,m_1}(\xi)\big|^2d\xi ds\Big)^{\frac12}\label{e219}\\
        &\Big(\sum_{m_2\in J_2}\int_{
				 |s|\lesssim \lambda^{-\gamma}
		}\int_{|\eta|\lesssim\la^{\lambda+\kappa}}\big|\widehat{D_{M(\overline{t})s}g}_{2,m_2}(\eta)\big|^2d\eta ds\Big)^{\frac12}\nonumber
    \end{align}
    By Plancherel's identity and the size of the support of $g_{j,m_j}$, we have
    \begin{align}
        &\sum_{m_1\in J_1}\int_{
				 |s|\lesssim \lambda^{-\gamma}
		}\int_{|B_R\xi|\lesssim\la^{\lambda+\kappa}}\big|\widehat{D_{B_Rs}g}_{1,m_1}(\xi)\big|^2 d\xi ds\nonumber\\
        \le&\sum_{m_1\in J_1}\int_{
				 |s|\lesssim \lambda^{-\gamma}
		}\|D_{B_Rs}g_{1,m_1}\|_2^2 ds\nonumber\\
        \lesssim&\la^{-2\gamma}.\label{e220}
    \end{align}
     Using that 
     $$
\big|M(\overline{t})\big|=\left|\begin{bmatrix}
    2\overline{t_1}&2\overline{t_2}\\
    3\overline{t_1}^2&3\overline{t_2}^2
\end{bmatrix}\right|=|6\overline{t_1}\ \overline{t_2}(\overline{t_1}-\overline{t_2})|\gtrsim A^{-3},
$$
by change of variables
$$
M(\bar t)s\to s,
$$
\eqref{e223} follows from the estimates \eqref{e219} and \eqref{e220}.
Similarly, we obtain 
     \begin{equation}\label{e221}
         \sum_{m_2\in J_2}\int_{
				 |s|\lesssim \lambda^{-\gamma}
		}\int_{|\eta|\lesssim\la^{\lambda+\kappa}}\big|\widehat{D_{M(\overline{t})s}g}_{2,m_2}(\eta)\big|^2d\eta ds\lesssim\la^{-2\gamma}.
     \end{equation}

By change of variables
$$
B_Rs\to s
$$
we obtain \eqref{e222} from \eqref{e219} and \eqref{e221}.

\end{proof}

We are now ready to estimate \eqref{e78}.
Applying Lemma~\ref{t67}, we obtain from  the definition of $g_{2,m_2}$ and \eqref{e47} with $\gamma'=\gamma+\kappa$ that
\begin{align}
      \eqref{e78}  \lesssim &A^{O(1)}\lambda^{\delta_4-\frac{\delta_2} 4}.\label{e224}
\end{align}
In conclusion, by \eqref{e088}, Proposition~\ref{t50}, Proposition~\ref{t35}, Proposition~\ref{t66} and \eqref{e224}, we  have
\begin{equation}\label{e225}
    \|T(f_{1},\,f_{2,\flat})\|_{L^1([0,1]^2)}\lesssim_K A^{O(1)}(\la^{-\delta+\delta_2}+\la^{\frac{-\kappa+5\delta_4}4+\delta_2}+\lambda^{\delta_4-\frac{\delta_2} 4}+\la^{-K}).
\end{equation}

To estimate $\|T(f_{1,\flat},\,f_{2,\sharp})\|_{L^1\zozo}$, we can repeat the above process with 
\begin{align*}
    \tilde g_{1,m}=\frac{1}{\la^{\delta_1}}f_{1,m,\flat},\\
    \tilde g_{2,m}=\frac{1}{\la^{\delta_2}}f_{2,m,\sharp}.
\end{align*}
It is straightforward  to check that $\|\tilde g_{j,m}\|_\infty\lesssim1$ and 
$$
\la^{-\delta_1-\delta_2}\|T(f_{1,\flat},\, f_{2,\sharp})\|_{L^1(\zozo)}\le \sum_{m\in J}\|T(\tilde g_{1,m_1},\, \tilde g_{2,m_2})\|_{L^1}.
$$

Repeat the above argument handling $\|T(f_{1},\,f_{2,\flat})\|_{L^1([0,1]^2)}$,
we obtain
\begin{equation}\label{e226}
    \|T(f_{1,\flat},\,f_{2,\sharp})\|_{L^1([0,1]^2)}\lesssim_K A^{O(1)}\la^{\delta_2}(\la^{-\delta+\delta_1}+\la^{\frac{-\kappa+5\delta_4}4+\delta_1}+\lambda^{\delta_4-\frac{\delta_1} 4}+\la^{-K}).
\end{equation}
where the only difference is that  we should use \eqref{e222} in Proposition~\ref{t67} to get
\begin{align*}
    &A^{O(1)}R^{\frac34}\lambda^{\frac\gamma2+\delta_4+\delta_1+\delta_2}\Big(\sum_{m_1\in J_1}\int_{
				 \bbr^2
		}\int_{|B_R\xi|\lesssim\la^{\lambda+\kappa}}\big|\widehat{D_{s}\tilde g}_{1,m_1}(\xi)\big|^2d\xi ds\Big)^{\frac14}\\
        \lesssim &A^{O(1)}\lambda^{\delta_4-\frac{\delta_1}4+\delta_2}.
\end{align*}

By \eqref{e225} and \eqref{e226}, we  find
\begin{equation}\label{e229}
    \|T_\sharp\|_{L^1([0,1]^2)}\lesssim \la^{-c}
\end{equation}\label{e230}
where $c>0$, if we take the appropriate $\delta_i$ and $\kappa$.

\section{Estimate for $\|T_\sharp\|_{L^1}$}\label{s5'}

In this section, we aim to prove
\begin{equation}\label{e078}
    \|T_\sharp\|_{L^1([0,1]^2)}\lesssim\lambda^{-c}
\end{equation}
for some absolute constant $c>0$. This combined with  \eqref{e230} and \eqref{e31}  completes the proof of Theorem~\ref{t25}.

We observe first that
\begin{align}
    &\|T_\sharp\|_{L^1(\zozo)}\notag\\
    =&\int_{\zozo}\Big|\int_{\bbr^2}\sum_{m\in J}\sum_{n_1=1}^{N_{1,m_1}}\sum_{n_2=1}^{N_{2,m_2}}h_{1,m_1,n_1}(x+B_R(t))h_{2,m_2,n_2}(x+P(t))\notag\\
    &e^{i\alpha_{1,m_1,n_1}\cdot(x+B_R(t))}
    e^{i\alpha_{2,m_2,n_2}\cdot(x+P(t))}\widetilde{\eta_{m_1}}(x+B_R(t))\widetilde{\eta_{m_2}}(x+P(t))r(t)dt\Big|dx\notag\\
    \le&\sum_{n_1=1}^{N_1}\sum_{n_2=1}^{N_2}\sum_{m\in J}\int_{\zozo}\Big|\int_{\bbr^2}e^{i\alpha_{1,m_1,n_1}\cdot(x+B_R(t))}e^{i\alpha_{2,m_2,n_2}\cdot(x+P(t))}H_{m,n}(x,\,t)dt\Big|dx\label{e077}
\end{align}
where 
\begin{align*}
    &H_{m,n}(x,\,t)\\
    =&h_{1,m_1,n_1}(x+B_R(t))h_{2,m_2,n_2}(x+P(t))\widetilde{\eta_{m_1}}(x+B_R(t))\widetilde{\eta_{m_2}}(x+P(t))r(t),
\end{align*}
$N_j\lesssim\la^{\delta_j}$, $j=1,2$ and $| B_R(\alpha_{1,m_1,n_1})|\lesssim\la,\,|\alpha_{2,m_2,n_2}|\sim\la$. We remark that $\spt H_{m,n}=A_m$. As $N_1N_2\lesssim\la^{\delta_1+\delta_2}$, it suffices to consider a fixed pair $(n_1,\,n_2)$, so we will ignore the indices $n_1$ and $n_2$ below. The problem is now reduced to showing
$$
\sum_{m\in J}\int_{\zozo}\Big|\int_{\bbr^2}e^{i\alpha_{1,m_1}\cdot(x+B_R(t))}e^{i\alpha_{2,m_2}\cdot(x+P(t))}H_{m}(x,\,t)dt\Big|dx
\lesssim \la^{-c}.
$$

 We write $\alpha_{1,m_1}=(\alpha_{1,m_1,1},\alpha_{1,m_1,2})\in \bbr^2$. 
We can take 
$$
(\overline{x_m},\,\overline{t_m})\in [(R^{-1}\lambda^{-\gamma}\bbz) \times (R^{-2}\lambda^{-\gamma}\bbz)\times(\lambda^{-\gamma}\bbz)\times (\lambda^{-\gamma}\bbz) ]
$$ 
such that
\begin{align*}
    (\overline {x_m},\overline{t_m})\in &\spt H_{m}+O(R^{-1}\lambda^{-\gamma}) \times O(R^{-2}\lambda^{-\gamma})\times O(\lambda^{-\gamma})\times O(\lambda^{-\gamma})\\
    =&A_m+O(R^{-1}\lambda^{-\gamma}) \times O(R^{-2}\lambda^{-\gamma})\times O(\lambda^{-\gamma})\times O(\lambda^{-\gamma}).
\end{align*}
 In particular, the set $\{\overline{t_m}\}$ is $\lambda^{-\gamma}$-separated, which implies 
 \begin{equation}\label{e1000}
     \#\overline{t_{m,i}}=O(\la^{2\gamma}).
 \end{equation}
 For any $(x,\,t)\in\spt H_{m}$, we have 
 $$
 |(R(x_1-\overline x_{m,1}),R^2(x_2-\overline x_{m,2}))|\lesssim A\la^{-\gamma},\quad |t-{\overline{t_{m}}}|\lesssim A\la^{-\gamma}.
 $$
 As a result, for $j=1,2$,
\begin{align*}
    &\alpha'_{1,m_{1},j}+2\alpha_{2,m_{2},1}t_j+3\alpha_{2,m_{2},2}t_j^2\\
    =&\alpha'_{1,m_{1},j}+2\alpha_{2,m_{2},1}{\overline{t_{m,j}}}+3\alpha_{2,m_{2},2}{\overline{t_{m,j}}}^2+O(A\la^{1-\gamma}),
\end{align*}
where we use the notation $\overline{t_m}=(\overline{t_{m,1}},\,\overline{t_{m,2}})$ and  $\alpha'_{1,m_1}=B_R\alpha_{1,m_1}$.

Let $n=B_Rm_1-m_2\in\bbr^2$. 
By \eqref{e82} and the definition of ${\overline{t_m}}$, 
\begin{equation}\label{e36}
    |n-n'|\lesssim\lambda^\gamma|{\overline{t_m}}-{\overline{t_{m'}}}|+O(1).
\end{equation}
We remark that the $\overline{t_m}$ here is a little different from $\overline{t_m}$ in last section.

For any $\beta\in\bbn^2$, we have 
$$
\|\partial^\beta\widetilde{\eta_{m_1}}\|_\infty\lesssim R^{\beta_1+2\beta_2}\la^{\gamma' |\beta|},\quad \|\partial^\beta\widetilde{\eta_{m_2}}\|_\infty\lesssim\la^{\gamma' |\beta|}
$$ 
and 
$$
\|\partial^\beta h_{1,m_1}\|_\infty\lesssim R^{\beta_1+2\beta_2}\la^{\gamma' |\beta|},\quad \|\partial^\beta h_{2,m_2}\|_\infty\lesssim\la^{\gamma' |\beta|}
$$
by Lemma~\ref{t18}.
Therefore, for any $\beta\in\bbn^2$, we get $|\partial^\beta_t H_{m}|\lesssim A^{O_\beta(1)}\la^{\gamma'|\beta|}$ since $|t|\lesssim1$ and $R\ge 1$.

Take $\sigma>0$ small enough such that
$$\sigma+\gamma'<1.$$
We define
\begin{equation}\label{eset}
    I=\{ m\in J :\Big|\alpha'_{1,m_{1},j}+2\alpha_{2,m_{2},1}{\overline{t_{m,j}}}+3\alpha_{2,m_{2},2}{\overline{t_{m,j}}}^2\Big|\le \la^{\gamma'+\sigma},\,j=1,\,2\}.
\end{equation}
 By the size of $\alpha_{1,m_1}$, we can get $|\alpha'_{1,m_1}|\lesssim\la$. For  $m\notin I$, 
 taking $\ep'<(2\gamma-1)/10$, we have 
 $\Big|\alpha'_{1,m_{1},j}+2\alpha_{2,m_{2},1}{{t_{j}}}+3\alpha_{2,m_{2},2}{{t_{j}}}^2\Big|\gtrsim \la^{\gamma'+\sigma}$ when $(x,t)\in \spt H_m$, which implies, by 
  integration by parts, that
$$
|\int_{\bbr^2}e^{i\alpha_{1,m_1}\cdot(x+B_Rt)}e^{i\alpha_{2,m_2}\cdot(x+P(t))}H_{m}(x,\,t)dt|\lesssim_K A^{O(1)}\la^{-K},
$$
which implies further 
\begin{equation}\label{e43}
    \sum_{m\in J/I}\int_{\zozo}|\int_{\bbr^2}e^{i\alpha_{1,m_1}\cdot(x+B_Rt)}e^{i\alpha_{2,m_2}\cdot(x+P(t))}H_{m}(x,\,t)dt|dx\lesssim_K A^{O(1)}\la^{-K}.
\end{equation}
It remains to estimate the contribution when $m\in I$, which actually is the main term.   Our key ingredient is a sublevel set estimate.

\begin{thm}\label{t100}
    For
    $$
    I=\{ m\in J :|\alpha'_{1,m_{1},j}+2\alpha_{2,m_{2},1}{\overline{t_{m,j}}}+3\alpha_{2,m_{2},2}{\overline{t_{m,j}}}^2|\le \la^{\gamma'+\sigma},\,j=1,\,2\}.
    $$
    defined in \eqref{eset},
    there exists $\tilde \de>0$ and $C_{\tilde\delta}>0$ such that 
    \begin{equation}\label{e089}
         \#I\le C_{\tilde\delta} R^3\la^{4\gamma-\tilde \delta}
        \end{equation}
       for all $R\ge1$ and $\lambda\ge 1$.
\end{thm}
Once we establish Theorem \ref{t100}, we have
\begin{equation}\label{e44}
    \sum_{m\in I}\int_{\zozo}|\int_{\bbr^2}e^{i\alpha_{1,m_1}\cdot(x+B_Rt)}e^{i\alpha_{2,m_2}\cdot(x+P(t))}H_{m}(x,\,t)dt|dx\lesssim\la^{-\tilde\delta}
\end{equation}
since $\spt H_m$ has measure $\lesssim A^4 R^{-3}\lambda^{-4\gamma}$.
Recalling \eqref{e077}, \eqref{e43} and \eqref{e44}, we obtain
$$
   \|T_\sharp\|_{L^1([0,1]^2)}\lesssim \lambda^{\de_1+\de_2-\tilde\de},
$$
which implies \eqref{e078}  by taking $\delta_1$, $\delta_2$ small enough.

\begin{proof}[Proof of Theorem~\ref{t100}]
    
As $\#I\le \#J\le CR^3\lambda^{4\gamma}$, \eqref{e089} holds when $\lambda\le \lambda_0$ for any finite $\lambda_0$. Taking $\lambda_0=\lambda_{\tilde\delta}$ large enough, we  assume below $\lambda\ge \lambda_{\tilde\delta}$. We will prove \eqref{e089} for this case by contradiction.

Suppose 
\begin{equation}\label{e085}
    \#I\ge R^3\la^{4\gamma-\tilde \delta},
\end{equation}
where $\tilde \delta>0$ is to be determined small constant. 
Using $\#m_1\le R^3\la^{2 \gamma}$, we apply the Cauchy-Schwarz inequality to obtain
$$
    R^3\la^{4\gamma-\tilde\delta}\le\sum_{m=(m_1,m_2)\in\mathbb Z^4}\chi_I(m)
    \le R^{3/2}\la^{\gamma}(\sum_{m_1}(\sum_{m_2}\chi_I(m_1,\,m_2))^2)^{\frac{1}{2}}.
$$
Therefore
\begin{align*}
    R^3\la^{6\gamma-2\tilde\delta}\le&\sum_{m_1}(\sum_{m_2}\chi_I(m_1,\,m_2))^2\\
     =&\sum_{m_1}(\sum_{m_2}\chi_I(m_1,\,m_2))(\sum_{m'_2}\chi_I(m_1,\,m'_2)).
\end{align*}
Let 
$$
D_{m_1}=\{(m_2,\,m_2')\in[0,\lambda^\gamma]^4:\,(m_1,\,m_2)\in I,\,(m_1,\,m'_2)\in I\}
$$
and
$$
D=\{m_1:\,|D_{m_1}|\ge\frac{\la^{4\gamma-2\tilde\delta}}{10}\}.
$$
Then $|D|\ge R^3\la^{2\gamma-2\tilde\delta}/2$. 

Fixing $m_1\in D$, we define
\begin{align*}
    &D^1_{m_1}=\{(m_2,\,m_2')\in[0,\lambda^\gamma]^4:\,|-3(\overline{t_{m',1}}+\overline{t_{m,1}})/R+6\overline{t_{m',1}}\overline{t_{m,1}}|\le\la^{-6\tilde\delta}\}\\
    &D^2_{m_1}=\{(m_2,\,m_2')\in[0,\lambda^\gamma]^4:\,|2/R^2+6\overline{t_{m',2}}\overline{t_{m,2}}|\le\la^{-6\tilde\delta}\}.
\end{align*}
Let $ D^0_{m_1}=D^1_{m_1}\bigcup D^2_{m_1}$, and  $D'_{m_1}=D_{m_1}\setminus D^0_{m_1}$.

We claim that
\begin{equation}\label{e086}
    |D'_{m_1}|\ge c_2\la^{4\gamma-2\tilde\delta},
\end{equation}
where $c_2>0$ is a small constant.
To prove this, we estimate $D^1_{m_1}$ first. 
Fixing $\overline {t_{m,1}}$
satisfying $|-3/R+6\overline{t_{m,1}}|\ge\la^{-3\tilde\delta}$, there are $O(\la^{\gamma-3\tilde\delta})$ possible $\overline{t_{m',1}}$ satisfying
$$
|-3(\overline{t_{m',1}}+\overline{t_{m,1}})/R+6\overline{t_{m',1}}\overline{t_{m,1}}|\le\la^{-6\tilde\delta}.
$$
Therefore
$$
\#\{(\overline{t_m},\overline{t_{m'}}):\ (m_2,m_2')\in D_{m_1}^1,\ |-3/R+6\overline{t_{m,1}}|\ge\la^{-3\tilde\delta}\}
\lesssim \la^{4\gamma-3\tilde\delta}
$$
by \eqref{e1000}.
It is easy to check that
$$
\#\{(\overline{t_m},\overline{t_{m'}}):\ (m_2,m_2')\in D_{m_1}^1,\ |-3/R+6\overline{t_{m,1}}|\le\la^{-3\tilde\delta}\}
\lesssim \la^{4\gamma-3\tilde\delta}
$$
since there are $O(\la^{\gamma-3\tilde\delta})$ possible $\overline{t_{m,1}}$ satisfying $|-3/R+6\overline{t_{m,1}}|\le\la^{-3\tilde\delta}$.
By \eqref{e10}, when $m_1$ and $\overline{t_m}$ are fixed,  there are $O(1)$ possible $m_2$. Using this observation and last two estimates, we obtain $|D^1_{m_1}|\lesssim\la^{4\gamma-3\tilde\delta}$. Similarly, $|D^2_{m_1}|\lesssim\la^{4\gamma-3\tilde\delta}$. Then \eqref{e086} follows when $\lambda$ is large enough.

Let
$$
E=\{ (m_2,\,m_2')\in[0,\lambda^\gamma]^4:\,|\{m_1:\,(m_2,\,m_2')\in D_{m_1}'\}|\ge \frac{c_2R^3\la^{2\gamma-2\tilde\delta}}4\}.
$$
We observe that
\begin{align*}
    \frac{c_2R^3\la^{6\gamma-2\tilde\delta}}2\le|E|R^3\la^{2\gamma}+|E^c|\frac{c_2R^3\la^{2\gamma-2\tilde\delta}}{4}\le|E|R^3\la^{2\gamma}+\frac{c_2R^3\la^{6\gamma-2\tilde\delta}}{4},
\end{align*}
where $E^c:=[0,\lambda^\gamma]^4\setminus E$,
thus $|E|\gtrsim{\la^{4\gamma-2\tilde\delta}}$. In particular, there exists $(m_2,\,m_2')\in E$ satisfying 
\begin{equation}\label{e079}
|m_2-m_2'|\gtrsim\la^{\gamma-\tilde\delta}.     
\end{equation}

Fix  $(m_2,\,m_2')\in E$ satisfying \eqref{e079}. Let 
$$
\widetilde{E}_{m_2,m_2'}=\{m_1\in\bbz^2:\,(m_2,\,m_2')\in D_{m_1}\}.
$$ 
For simplicity, we abbreviate $\tilde E_{m_2,m_2'}$ by $\tilde E$. By definition, we have 
$|\widetilde{E}|\gtrsim{R^3\la^{2\gamma-2\tilde\delta}}.$ Denote $m=(m_1,\,m_2)$ and $m'=(m_1,\,m_2')$. By the definition of $D'_{m_1}$ we have $m,\,m'\in I$,
$$
|-3(\overline{t_{m',1}}+\overline{t_{m,1}})/R+6\overline{t_{m',1}}\overline{t_{m,1}}|\ge\la^{-6\tilde\delta}\textit{ and }|2/R^2+6\overline{t_{m',2}}\overline{t_{m,2}}|\ge\la^{-6\tilde\delta}.
$$
Since $|m_2-m_2'|\ge\frac{\la^{\gamma-\tilde\delta}}{4}$, \eqref{e36} implies that $|\overline{t_m}-\overline{t_{m'}}|\gtrsim\la^{-2\tilde\delta}.$ So we have  either 
$$
|\overline t_{m,1}-\overline t_{m',1}|\gtrsim\la^{-2\tilde\delta}
$$
or 
$$
|\overline t_{m,2}-\overline t_{m',2}|\gtrsim\la^{-2\tilde\delta}.
$$
Moreover, we have either 
\begin{equation}\label{e080}
    \#\{m_1\in\widetilde{E}:\,|\overline t_{m,1}-\overline t_{m',1}|\gtrsim\la^{-2\tilde\delta}\}\gtrsim R^3\la^{2\gamma-2\tilde\delta}
\end{equation}
or 
\begin{equation}\label{e93}
    \#\{m_1\in\widetilde{E}:\,|\overline t_{m,2}-\overline t_{m',2}|\gtrsim\la^{-2\tilde\delta}\}\gtrsim R^3\la^{2\gamma-2\tilde\delta}.
\end{equation}
Without loss of generality, we assume that \eqref{e080} holds, i.e. the set
$$
E':=\{m_1\in\widetilde{E}:\,|\overline t_{m,1}-\overline t_{m',1}|\gtrsim\la^{-2\tilde\delta}\}
$$
satisfies $\#E'\gtrsim R^3\la^{2\gamma-2\tilde\delta}$. For any $m_1\in E'$, we obtain from \eqref{e081},  the definition of $I$,  and the fact that $m,m'\in I$, the following inequalities:
\begin{align}
    &\Psi_1(\overline{t_m})-\Psi_1(\overline{t_{m'}})=\la^{-\gamma}(-m_{2,1}+m'_{2,1})+O(\la^{-\gamma})\label{e37};\\
    &\Psi_2(\overline{t_m})-\Psi_2(\overline{t_{m'}})=\la^{-\gamma}(-m_{2,2}+m'_{2,2})+O(\la^{-\gamma})\label{e38};\\
    &|2\alpha_{2,m_{2},1}\overline{t_{m,1}}+3\alpha_{2,m_{2},2}\overline{t_{m,1}}^2-2\alpha_{2,m'_{2},1}\overline{t_{m',1}}-3\alpha_{2,m'_{2},2}\overline{t_{m',1}}^2|\lesssim\la^{\sigma+\gamma'}\label{e39};\\
    &|2\alpha_{2,m_{2},1}\overline{t_{m,2}}+3\alpha_{2,m_{2},2}\overline{t_{m,2}}^2-2\alpha_{2,m'_{2},1}\overline{t_{m',2}}-3\alpha_{2,m'_{2},2}\overline{t_{m',2}}^2|\lesssim\la^{\sigma+\gamma'}\label{e40}.
\end{align}
We recall that $\Psi(t)=(\Psi_1(t),\Psi_2(t))=B_Rt-P(t)$.

Next, we will construct a subset $E''$ of $E'$, whose cardinality is both large and small, which will finally lead to a contradiction.

The relation \eqref{e40} roughly says that the number of $(\overline{t_{m,2}}, \overline{t_{m',2}})$ is small, so we can use pigeonholing to construct a subset $E''\sset E'$ with large cardinality.
From \eqref{e1000} we can get $\#\{\overline{t_{m,2}}\}\le \la^{\gamma}$, so there exists $t_0\in\bbr$ satisfying that
\begin{equation}\label{e082}
\#\{m_1\in E':\,\overline{t_{m,2}}=t_0\}\gtrsim R^3\la^{\gamma-2\tilde\delta}.
\end{equation}
We claim that
\begin{align}
    &\#\{\overline{t_{m',2}}:\, |2\alpha_{2,m_{2},1}t_0+3\alpha_{2,m_{2},2}t_0^2-2\alpha_{2,m'_{2},1}\overline{t_{m',2}}-3\alpha_{2,m'_{2},2}\overline{t_{m',2}}^2|\notag\\
    \lesssim&\la^{\sigma+\gamma'}\}\lesssim\la^{\gamma+\frac{\sigma+\gamma'-1}{2}}.\label{e083}
\end{align}
It follows from \eqref{e082} and \eqref{e083} that
there exists $t_0'$ such that the set
$$
E'':=\{m_1\in E':\,\overline{t_{m,2}}=t_0,\,\overline{t_{m',2}}=t'_0\}
$$
satisfying
\begin{equation}\label{e41}
    \#E''\gtrsim R^3\la^{\frac{1-\gamma'-\sigma-4\tilde\delta}{2}}.
\end{equation}
We verify \eqref{e083} below.
If $|\alpha_{2,m'_{2},2}|\ge\frac{\la}{4}$, then, by \eqref{e40} and that $\{\overline{t_m}\}$ is $\lambda^{-\gamma}$-separated
\begin{align*}
     &\#\{\overline{t_{m',2}}:\, |2\alpha_{2,m_{2},1}t_0+3\alpha_{2,m_{2},2}t_0^2-2\alpha_{2,m'_{2},1}\overline{t_{m',2}}-3\alpha_{2,m'_{2},2}\overline{t_{m',2}}^2|\lesssim\la^{\sigma+\gamma'}\}\\
     \lesssim&\la^{\gamma+\frac{\sigma+\gamma'-1}{2}}.
\end{align*}
If $|\alpha_{2,m'_{2},2}|\le\frac{\la}{4}$, because $|\alpha_{2,m'_{2}}|\ge\frac{\la}{2}$ due to the Fourier support of $f_2$, we have $|\alpha_{2,m'_{2},1}|\ge\frac{\la}{3}$. Thus, we have
\begin{align*}
    &\#\{\overline{t_{m',2}}:\, |2\alpha_{2,m_{2},1}t_0+3\alpha_{2,m_{2},2}t_0^2-2\alpha_{2,m'_{2},1}\overline{t_{m',2}}-3\alpha_{2,m'_{2},2}\overline{t_{m',2}}^2|\lesssim\la^{\sigma+\gamma'}\}\\
    \lesssim&\la^{\gamma+\sigma+\gamma'-1}.
\end{align*}
\eqref{e083} is now proved.

On the other hand, \eqref{e37} and \eqref{e38} indicate that $\#E''$ is small. To explain this, let us fix $\overline{t_{m,2}}=t_0,\,\overline{t_{m',2}}=t'_0$, and define
$$
\Phi(s,s'):= \Psi(s,t_0)-\Psi(s',t_0'),
$$
where $\Psi$ is defined by \eqref{e084}. Then 
\begin{align*}
    |\det J(\Phi)|=&\left|\begin{bmatrix}
    R^{-1}-2s &-R^{-1}+2s'\\
    -3s^2 &3s'^2
        \end{bmatrix}\right|\\
        =&|s-s'||-3R^{-1}(s+s')+6ss'|\\
        \gtrsim&\la^{-8\tilde\delta},
\end{align*}
when  $(s,s')\in S:=\{(s,s')\in[0,1]^2:\ |-3R^{-1}(s+s')+6ss'|\ge \lambda^{-6\tilde\de},\ |s-s'|\ge \lambda^{-2\tilde\de}\}$. This combined with Lemma~\ref{t5} implies that 
$\{(s,s')\in S:\ \Phi(s,s')=a+O(\lambda^{-\gamma})\}$ is contained in finitely many squares of length $\sim \lambda^{-\gamma+8\tilde\delta}$.
As $m_2$ and $m_2'$ are fixed, we obtain from this observation, \eqref{e37}, and  \eqref{e38} that $(\overline{t_{m,1}},\,\overline{t_{m',1}})$ lies in several squares with length $O(\la^{-\gamma+8\tilde\delta})$. 
In particular,  $\overline{t_{m,1}}$ lies in several intervals of length $O(\la^{-\gamma+8\tilde\delta})$, which combined with \eqref{e36} yields 
\begin{equation}\label{e42}
     \#E''
     \lesssim R^3\la^{8\tilde\delta}.
\end{equation}
as $m=(m_1, m_2)$ and $m_2$ is fixed.

Combining \eqref{e41} and \eqref{e42}, we obtain $20\tilde\delta\ge1-\gamma'-\sigma$, which leads to a contradiction if we take $0<\tilde\delta\le\frac{1-\gamma'-\sigma}{24}$.

In conclusion, the assumption \eqref{e085} fails, and we have $\# I\le\la^{4\gamma-\tilde\delta}$. This completes the proof of Theorem~\ref{t100}.
\end{proof}

\begin{rmk}
    We discuss parameters in this note.
    From \eqref{e15}, \eqref{e225}, and \eqref{e226}, we can get 
   \begin{align*}
       c=\min\{&\delta-\delta_2,\,\frac{\kappa-5\delta_4}4-\delta_2,\,\delta_4-\frac{\delta_2} 4,\\
       &\delta-\delta_2-\delta_1,\,\frac{\kappa-5\delta_4}4-\delta_1-\delta_2,\,\frac{\delta_1} 4-\delta_4-\delta_2\}
   \end{align*}
    in \eqref{e229}. So if we take $\delta_1\le\frac{\min\{\delta,\kappa\}}{100}$, $\delta_2\le\frac{\delta_1}{100}$ and $\delta_4\le\frac{\delta_2}{100}$, we can find $c>0$ such that 
    $$
    \|T_{\flat}\|_{L^1(\zozo)}\lesssim_l\la^{-c}.
    $$
    During this section we only need $\delta_1+\delta_2<\tilde\delta\le\frac{1-\gamma'-\sigma}{24}$ to ensure \eqref{e078}.
    For example, we can take $\kappa=\de=\gamma-\frac{1}{2}=\frac{1}{100}$, $\tilde\delta=10^{-3}$, $\delta_1=10^{-4}$, $\delta_2=10^{-6}$, $\ep'=\delta_3=\delta_4=\sigma=10^{-8}$.
\end{rmk}

\appendix

\section{ Some Roth Theorems in $\bbr^2$}

\begin{thm}\label{t1}
	Let $\ep\in(0,\ \frac{1}{2})$. Then there exists a constant $\de(\ep)>0$ such that the following holds. If $E\subseteq [0,\, 1]^{2}$  is a measurable set of Lebesgue measure greater than $\ep$, then there exist
    $$
    (x,y),\, (x+t,\,y+s^2),\, (x+s,\, y+t^2)\in E
    $$ with $t,\, s>\delta(\ep)$.
\end{thm}
Let $v\ge0$ be an even smooth function which is supported in $[-2,\,2]$, constant on $[-1,\,1]$, monotone on $[1,\,2]$ and normalized such that $\|v\|_1=1$. For $k\in \bbz^+$, define $v_k(x)=2^{k}v(2^kx)$. The following is a reformulation of \cite[Theorem 5]{Christ2020c}.

\begin{prop}[{\cite[Theorem 5]{Christ2020c}}]\label{t2}
		
        There exists an absolute constant $\gamma>0 $, independent of $m\in\bbn$, such that for every Schwartz function $g$ with 
        $$
        \spt(\wh{g})\sset\bbr\times \big([2^{m-1},2^m]\cup[-2^m,-2^{m-1}]\big),
        $$
        we have
	\begin{equation}
		\Big\|\int_{\bbr}f(x+t,\, y)g(x,\, y+t^2)v_{l}(t)dt\Big\|_{L^{1}({\zozo})}\le C_{l}2^{-\gamma m}\|f\|_{2}\|g\|_{2}\,,
	\end{equation}
	where $C_{l}\le2^{\gamma_{0}l}$ for some $\gamma_{0}>0$ .
\end{prop}

For $f,\, g\in L^2(\bbr^2)$, we define
\begin{align*}
    &f\ast_1g(x,\,y)=\int_\bbr f(x-t,y)g(t,y)dt,\\
     &f\ast_2g(x,\,y)=\int_\bbr f(x,y-t)g(x,t)dt.
\end{align*}
 In Section 5.1 in \cite{Christ2020c}, it was shown that the Proposition~\ref{t2} implies the following estimate.

\begin{lm}\label{l2}
    Let $6<k+5<k'<k''$ be natural numbers, and let $f_0,\,f_1,\,f_2$ be non-negative $1$-bounded functions supported in $\zozo$. Define
    $$
    I=\int f_0(x,\,y)f_1(x+t,\,y)f_2(x,\,y+t^2)v_{k'}(t)dxdydt.
    $$
    Then there exists $\sigma,\,c>0$ such that
    \begin{align*}
        |I|
        \ge&\int_{\zozo}f_0(f_1\ast_1v_k)(f_2\ast_2v_k)\\
        &-c(2^{2\sigma k'-\sigma k''}+2^{k'-k''}+2^{k-k'}\\
        &+\|f_2\ast_2v_{k''}-f_2\ast_2v_k\|_2+\|f_1\ast_1v_{k''}-f_1\ast_1v_k\|_2).
    \end{align*}

\end{lm}

We  explain how Theorem \ref{t1} follows from Theorem \ref{t2}. The reduction is similar to that in \cite{Christ2020c}.

\begin{proof}[Proof of Theorem~\ref{t1}]

To prove Theorem \ref{t1}, it suffices to prove that for any measurable function $f$ on $\zozo$ satisfying $0\le f\le 1$ and $\int_{\bbr^2} f(x)dxdy\ge\ep$, we have 
$$
\int_{\bbr^2}\int_{\zozo}f(x)f(x+t,y+s^2)f(x+s,y+t^2)dxdydtds\ge\delta(\ep).
$$

Recalling that
    $$
Q_{t}f(x,y)=\f{1}{4t^{2}}\int_{|(u\com v)-(x\com y)|\le t}f(u\com v)dudv,
$$ 
    we have that
    $$
    (f\ast_1v_k\ast_2v_k)(x,\,y)\gtrsim Q_{2^k}f(x,\,y)\ge0.
    $$
    From Lemma~\ref{l1}, we can get that for $f(x)\ge0$ and $1\le k_1\le k_2$,
    \begin{align}
        &\int_{\zozo}f(x,\,y)(f\ast_1v_{k_1}\ast_2v_{k_1})(x,\,y)(f\ast_1v_{k_2}\ast_2v_{k_2})(x,\,y)dxdy\nonumber\\
        \gtrsim&(\int_{\zozo}f(x,\,y)dxdy)^3\nonumber\\
        \gtrsim&\ep^3.\label{e707}
    \end{align}

Fix integers $6<k+5<k'<k''$ to be determined later. Because $\|v_k\|_\infty\lesssim2^{k}$, we have
\begin{align}
    &2^{2k'}\int_{\bbr^2}\int_{\zozo}f(x,\,y)f(x+t,y+s^2)f(x+s,y+t^2)dxdydtds\nonumber\\ 
    \gtrsim&\int_{\bbr^2}\int_{\zozo}f(x,\,y)f(x+t,y+s^2)f(x+s,y+t^2)v_{k'}(t)v_{k'}(s)dxdydtds.\label{e67}
\end{align}
Fix $s$. By Lemma~\ref{l2}, we obtain 
\begin{align}
	&\int_{\bbr}\int_{\zozo}f(x,\,y)f(x+t,y+s^2)f(x+s,y+t^2)v_{k'}(t)dxdydt\nonumber\\ \gtrsim&\int_{\zozo}f(x,\,y)(f\ast_1v_k)(x,\,y+s^2)(f\ast_2v_k)(x+s,y)dxdy \nonumber\\
    &-c(2^{2\sigma k'-\sigma k''}+2^{k'-k''}+2^{k-k'}\nonumber\\
    &+\|f\ast_2v_{k''}-f\ast_2v_k\|_2+\|f\ast_1v_{k''}-f\ast_1v_k\|_2).\label{e3}
\end{align}

Integrating \eqref{e3} with $v_{k'}(s)$ and applying Lemma~\ref{l2} again, we  get 
\begin{align}
	2^{2k'}&\int_{\bbr^2}\int_{\zozo}f(x,\,y)f(x+t,y+s^2)f(x+s,y+t^2)dxdydtds\nonumber\\
    \gtrsim&\int_{\zozo}f(x,\,y)(f\ast_1v_k\
    \ast_2v_k)(x,\,y)(f\ast_2v_k\ast_1v_k)(x,y)dxdy \nonumber\\
    &-c(2^{2\sigma k'-\sigma k''}+2^{k'-k''}+2^{k-k'}\nonumber\\
    &+\|f\ast_2v_{k''}-f\ast_2v_k\|_2+\|f\ast_1v_{k''}-f\ast_1v_k\|_2)\nonumber\\
    &-c(2^{2\sigma k'-\sigma k''}+2^{k'-k''}+2^{k-k'}\nonumber\\
    &+\|f\ast_1v_k\ast_2v_{k''}-f\ast_1v_k\ast_2v_k\|_2+\|f\ast_2v_k\ast_1v_{k''}-f\ast_2v_k\ast_1v_k\|_2).\label{e102}
\end{align}
In deriving \eqref{e102}, we also use the fact that $\|v_{k'}\|=1$ and \eqref{e67}. 
By Young's inequality and the fact that $\|v_k\|_1\sim1$, we have 
\begin{align*}
    \|f\ast_1v_k\ast_2v_{k''}-f\ast_1v_k\ast_2v_k\|_2&\lesssim\|f\ast_2v_{k''}-f\ast_2v_k\|_2,\\
    \|f\ast_2v_k\ast_1v_{k''}-f\ast_2v_k\ast_1v_k\|_2&\lesssim\|f\ast_1v_{k''}-f\ast_1v_k\|_2.
\end{align*}
Substituting these into \eqref{e102}, we get 
\begin{align}
    &2^{k'}\int_{\bbr^2}\int_{\zozo}f(x)f(x+t,y+s^2)f(x+s,y+t^2)dxdydtds\nonumber\\\gtrsim&\int_{\zozo}f(x,\,y)(f\ast_1v_k\
    \ast_2v_k)(x,\,y)(f\ast_2v_k\ast_1v_k)(x,y)dxdy \nonumber\\
    &-2c(2^{2\sigma k'-\sigma k''}+2^{k'-k''}+2^{k-k'}\nonumber\\
    &+\|f\ast_2v_{k''}-f\ast_2v_k\|_2+\|f\ast_1v_{k''}-f\ast_1v_k\|_2).\label{e4}
\end{align}
By \eqref{e707} and the discussion in Section 5.1 of \cite{Christ2020c}, we conclude that 
$$
\int_{\bbr^2}\int_{\zozo}f(x)f(x+t,y+s^2)f(x+s,y+t^2)dxdydtds\ge\delta(\ep),
$$ 
which completes the proof of Theorem~\ref{t1}.

\end{proof}
\begin{rmk}
    By a similar argument, we can get an analogue of Theorem~\ref{t1} with
    $$
    (x,y),\, (x+t,\,y+s^2),\, (x+s,\, y+t^2)
    $$
 replaced by
    $$
    (x,y),\, (x+t,\,y+s),\, (x+s^2,\, y+t^2)
    $$
    or 
    $$
    (x,y),\, (x+t,\,y),\, (x+s,\, y+t^2).
    $$
\end{rmk}

\end{document}